\documentclass{amsart}
\usepackage[matrix,arrow,tips,curve]{xy}
\usepackage[latin1]{inputenc} %ADDED
\usepackage[T1]{fontenc} %ADDED
\usepackage[activeacute]{babel} %ADDED
\usepackage{amssymb} %ADDED
\usepackage{dsfont} %ADDED
\usepackage[mathscr]{euscript} %ADDED
\usepackage{old-arrows} %ADDED
\usepackage{tikz} %ADDED
\usepackage{tikz-cd} %ADDED
\usepackage{nicefrac} %ADDED
\usepackage{faktor} %ADDED
\usepackage[title,titletoc]{appendix} %ADDED
\usepackage{listings} %ADDED
\usepackage{amsmath,amsthm,amsfonts}
\usepackage{mathtools}
\usepackage{enumitem}
\usepackage{geometry}
\usepackage{xspace}
\usepackage{hyperref}
\usepackage{xcolor}
\usepackage{listings} %ADDED

\hypersetup{
	colorlinks=true,
	linkcolor=blue,
	citecolor=blue,
	urlcolor=blue,
	pdftitle={Your paper title},
	pdfauthor={Author Name}
}

\newtheorem{theorem}{Theorem}[section]
\newtheorem{lemma}[theorem]{Lemma}
\newtheorem{proposition}[theorem]{Proposition}
\newtheorem{corollary}[theorem]{Corollary}

\theoremstyle{definition}
\newtheorem{definition}[theorem]{Definition}
\newtheorem{example}[theorem]{Example}

\theoremstyle{remark}
\newtheorem{remark}[theorem]{Remark}

\newcommand\sbullet[1][.5]{\mathbin{\vcenter{\hbox{\scalebox{#1}{$\bullet$}}}}}
\newcommand{\id}{\textrm{id}}
\newcommand{\longthanks}[1]{%
	\par\bigskip
	\noindent\textbf{Acknowledgments.}~#1\par
}

\numberwithin{equation}{section}

\makeatletter
\@namedef{subjclassname@2020}{\textup{2020} Mathematics Subject Classification}
\makeatother

\begin{document}
\title{On algebro-geometric quotients of torus-invariant subvarieties of the flag variety}

%    Remove any unused author tags.

%    author one information
\author{Luis Yair Meza-P\'erez}
\address{Luis Yair Meza-P\'erez, Divisi\'on Acad\'emica de Ciencias B\'asicas, Universidad Ju\'arez Aut\'onoma de Tabasco, Carretera Cunduac\'an-Jalpa de M\'endez KM. 1, Col. La Esmeralda, C.P. 86690, Cunduac\'an, Tabasco, M\'exico}
\curraddr{}
\email{\textcolor{blue}{matematico\_meza@hotmail.com}}
\thanks{The first author was supported by the Consejo Nacional de Ciencia y Tecnolog\'ia (CONACYT) through Mexican scholarship No. 637186. He wishes to thank the Centro de Investigaci\'on en Matem\'aticas, A.C. (CIMAT) for its hospitality and the financial, logistical, and technical support provided during his academic research stays.}

%    author two information
\author{Pedro L. del \'Angel R.}
\address{Pedro L. del \'Angel R., Departamento de Matem\'aticas B\'asicas, Centro de Investigaci\'on en Matem\'aticas, A.C., Jalisco S/N, Col. Valenciana, C.P. 36023, Guanajuato, Gto., M\'exico}
\curraddr{}
\email{\textcolor{blue}{luis@cimat.mx}}
\thanks{}

%    author two information
\author{Carlos Pompeyo-Guti\'errez}
\address{Carlos Pompeyo-Guti\'errez, Divisi\'on Acad\'emica de Ciencias B\'asicas, Universidad Ju\'arez Aut\'onoma de Tabasco, Carretera Cunduac\'an-Jalpa de M\'endez KM. 1, Col. La Esmeralda, C.P. 86690, Cunduac\'an, Tabasco, M\'exico}
\curraddr{}
\email{\textcolor{blue}{carlos.pompeyo@ujat.mx}}
\thanks{}

%    author two information
\author{Miguel Angel Dela-Rosa}
\address{Miguel Angel Dela-Rosa, División Acad\'emica de Ciencias B\'asicas, SECIHTI-UJAT, Carretera Cunduac\'an-Jalpa de M\'endez KM. 1, Col. La Esmeralda, C.P. 86690, Cunduac\'an, Tabasco, M\'exico}
\curraddr{}
\email{\textcolor{blue}{madelarosaca@secihti.mx}}
\thanks{}

\subjclass[2020]{Primary 14M15, 14L30; Secondary 05E14, 14L24, 14N15.}

\keywords{Algebraic Geometry, Matroids, Quotients, $T$-varieties, Geometric Invariant Theory}

\date{\today}

\dedicatory{The first author dedicates this paper to the loving memory of Mat. Rodolfo Conde del \'Aguila, an exceptional teacher, mentor, and friend, whose passion for mathematics inspired me to pursue this beautiful and fascinating discipline. With deep gratitude, a warm embrace to Heaven, Maestro.}

\begin{abstract}
	In this paper, we study the subvarieties of a complex flag variety that are invariant under the action of a maximal torus. Using combinatorial techniques derived from matroid theory, we introduce a decomposition of this variety into affine, locally closed subsets, which we refer to as \textit{thin Schubert cells}, each indexed by an element of a Cartesian product of matroids. We also show that the set of orbits for each of these thin Schubert cells under the action of the torus is, in fact, an orbit space, which gives rise to topologically trivial fiber bundles. As a consequence of this, we prove the existence of algebro-geometric quotients in the sense of Mumford's \textit{Geometric Invariant Theory}. The main result of this work is the existence of a surjective map from the set of geometric quotients of thin Schubert cells to the invariant scheme-theoretic points of a complex flag variety, which allows us to decompose it in terms of such quotients and also define a map from the set of these geometric quotients to the invariant homology of the variety. Finally, we give a complete description of the thin Schubert cells for the special case of the flag variety $\mathds{F}_{1<n-1}(\mathds{C}^n)$ and derive explicit counting formulas.
\end{abstract}

\maketitle

\longthanks{We would like to express our sincere gratitude to \emph{Dr. Cristhian Garay L\'opez,} research professor at the Centro de Investigaci\'on en Matem\'aticas, A.C. (CIMAT), for his insightful discussions during the early stages of the preparation of this paper. We also thank \emph{Professor June E Huh}, from the Department of Mathematics at Princeton University, for his valuable feedback on the first version of this paper. His observations allowed us to provide an additional characterization of thin Schubert cells in the flag variety $\mathds{F}_{1<n-1}(\mathds{C}^n)$ (see \S\ref{sec5}), now in terms of flag matroids.}

\bigskip
\tableofcontents
\bigskip

\section{Introduction}
    \textit{Homogeneous spaces} are nonempty algebraic varieties endowed with a transitive action of an algebraic group (reductive, semisimple). In simple terms, they can be thought of as spaces that appear identical regardless of how one moves through them via the group action. More precisely, any such space is isomorphic to a quotient of the form $G/P$, where $G$ is an algebraic group and $P$ is a closed subgroup. Ehresmann \cite{Ehr34} provides a rigorous and foundational description of these spaces in terms of transitive actions of Lie groups and shows that the Grassmannian $G(d,n)$ can be realized as the quotient $\mathrm{GL}_n(\mathds{C})/P$ for a suitable parabolic subgroup $P$. He also establishes a cellular decomposition of $G(d,n)$ into subsets known as \textit{Schubert cells}, thereby generalizing the ideas previously introduced by Hermann Schubert in 1879. Subsequently, Monk \cite{Mon59} extends some of Ehresmann's results to complete flag varieties and conducts a thorough study of their topological and algebro-geometric properties. Furthermore, he gives a detailed description of the celullar structure of these varieties in terms of Schubert cells. Building upon these foundational developments, Gelfand et al. \cite{Gel-et87} provide three equivalent stratifications of the Grassmannian; the first one uses matroid techniques, and the third one, Schubert cells. In a subsequent work, Gel'fand and Serganova \cite{GS87} generalize this approach to arbitrary homogeneous spaces $G/P$, {describing} it in what they define as \textit{thin \textup{[}Schubert\textup{]} cells}. One can identify a complex flag variety $X$ with the quotient $\mathrm{GL}_n(\mathds{C})/P$ and explicitly describe the parabolic subgroup $P$ as the isotropy group of a fixed flag. It is possible to show that the cohomology classes of the Schubert varieties in $X$ form an additive basis of the cohomology ring $H^*(X)$ (see Brion \cite{Bri05}).
    
    Our first goal is to generalize to complex flag varieties two results that Elizondo-Fink-Garay \cite{EFG25} obtain for the complex Grassmannian $G(d,n)$ using combinatorial methods, namely: \textit{\textbf{\textup{(}1\textup{)}}} realize the thin Schubert cell $G_M$ as an affine and locally closed subset (see \S2.2), and \textit{\textbf{\textup{(}2\textup{)}}} obtain a decomposition of $G(d,n)$ as a disjoint union of these thin cells (Proposition 2.4). For this, observe that there exists a natural action of the maximal algebraic torus $T\cong(\mathds{C}^*)^n$ on the Cartesian product of Grassmannians $G(d_1,n)\times\cdots\times G(d_k,n)$ under which $\mathds{F}_{d_1<\cdots<d_k}(\mathds{C}^n)$ is a $T$-invariant subset. This enables us to construct, in Lemma \ref{fmaffine}, the analogue of a thin Schubert cell on flag varieties, and to prove that it is an affine and locally closed set, and consequently to establish, in Proposition \ref{fdecomp}, a decomposition of the flag variety as a disjoint union of such cells\footnote{Our description of thin Schubert cells coincides with flag matroids in the case of the variety $\mathds{F}_{1<n-1}(\mathds{C}^n)$, but the characterization we provide allows us to determine quickly that such thin cells are affine schemes.}. The above leads us to the main results of this work:
    
        \begin{itemize} 
        	\item In Theorem \ref{fm-ugq}, we demonstrate that the orbit space of a thin Schubert cell under the action of a subtorus is a universal geometric quotient, and in particular, a universal categorical quotient.
        	\item In Theorem \ref{subvar-ugq}, we prove that the intersection of a $T$-invariant subvariety with a thin cell, modulo the action of a subtorus, admits a universal geometric quotient and, hence, a universal categorical quotient.
        	\item In Theorem \ref{mainthm}, we provide a surjective map from the disjoint union of the universal geometric quotients to $\mathds{F}_{d_1<\cdots<d_k}(\mathds{C}^n)^T$, which denotes the scheme-theoretic set of $T$-invariant points of $\mathds{F}_{d_1<\cdots<d_k}(\mathds{C}^n)$.
        \end{itemize}		    
    
    As a consequence of Theorem \ref{mainthm}, we establish in Corollary \ref{vmdecomp} a decomposition of scheme-theoretic points of $\mathds{F}_{d_1<\cdots<d_k}(\mathds{C}^n)$ in terms of the disjoint union of the images under the surjective map, and in Corollary \ref{hom-map} we further provide a map from the disjoint union of the thin cells to the $T$-invariant homology.
    
    Finally, in Proposition \ref{charthincells} we classify the thin Schubert cells of the flag variety $\mathds{F}_{1<n-1}(\mathds{C}^n)$ for $n\geq 3$, establish a purely combinatorial criterion characterizing the nonempty ones, and consequently obtain a complete description of the remaining cases. We then determine the dimension of each class and compute the number of thin Schubert cells contained in each.

\section{Preliminaries}\label{sec2}
    First, we introduce some conventions and preliminaries from Combinatorics and Grassmannians.\\
    
    \noindent\textbf{Conventions.} We will consider a strictly increasing sequence $0<d_1<\cdots<d_k<n$ of integers; $[n]$ will denote the set $\{1,\ldots,n\}$ with its usual order, and $\displaystyle{[n]\choose d_\ell}$ will denote the set of all subsets of cardinality $d_\ell$  of $[n]$ with the following partial order: if $I=(i_1,\ldots,i_{d_\ell}),J=(j_1,\ldots,j_{d_\ell})\in\displaystyle{[n]\choose d_\ell}$ where $i_1<\cdots<i_{d_\ell}$ and $j_1<\cdots<j_{d_\ell}$, we will say that $I\leq J$ if and only if $i_s\leq j_s$ for all $1\leq s\leq d_\ell$. Finally, $\displaystyle{[n]\choose d_1,\ldots,d_k}$ will denote the Cartesian product $\displaystyle{[n]\choose d_1}\times\cdots\times\displaystyle{[n]\choose d_k}$, and $\mathbf{W},\mathbf{I}$ will denote the $k$-tuples $(W_1,\ldots,W_k)$ and $(I_1,\ldots,I_k)$ respectively.
    
    A \textbf{decomposition} of a topological space $X$ is a family $\{X_i\}_{i\in I}$ of pairwise disjoint and locally closed subsets of $X$, satisfying $X=\bigsqcup_{i\in I}X_i$.
    
\subsection{Matroids and the Grassmannian variety $G(d,n)$}\label{subsec2.1}
    Matroids are combinatorial objects useful for capturing the discrete linear algebraic information of an element $W_{\sbullet}\in\mathds{F}_{d_1<\cdots<d_k}(\mathds{C}^n)$, and they serve as the indexing objects for the decomposition of the flag variety into thin cells, as stated in Proposition \ref{fdecomp}, which is fundamental to this work. In this section, we introduce them by extending the notation provided in \cite[\S2.1, p. 3]{EFG25}, which will serve as our standard reference for the purposes of this paper.

    A \textbf{matroid} of \textbf{rank} $d$, on $n$ elements is a pair $M=([n],\mathcal{B})$, where $\mathcal{B}\subset{[n]\choose d}$ is a nonempty family which satisfies the following exchange property:
        \begin{equation}\label{exchange}
        	\mbox{for every\ } I,J\in \mathcal{B} \mbox{\ and $i\in I\setminus J$, there exists\ } j\in J\setminus I\ \mbox{such that}\ (I\setminus\{i\})\cup\{j\}\in \mathcal{B}.
        \end{equation}
            
    The members of $\mathcal{B}$ are the \textbf{bases} of $M$. We denote by $\EuScript{M}^d_n$ the set of matroids of rank $d$ on $n$ elements, and we define $\EuScript{M}^{d_1,\ldots,d_k}_n := \EuScript{M}^{d_1}_n \times\cdots\times \EuScript{M}^{d_k}_n$ as their Cartesian product, whose elements we refer to as \textbf{plurimatroids}.
    
    The subgroup $T<\mathrm{GL}_n(\mathds{C})$ of invertible diagonal matrices with entries {\tiny }in $\mathds{C}$ is a \emph{maximal torus} acting on the Grassmannian $G(d,n)$ as follows.
    
    \begin{definition}
    	The map
    	    \begin{equation}\label{tactiongrass}
    	    	\begin{array}{cccc}
    	    		\sbullet:&G(d,n)\times T&\longrightarrow&G(d,n)\\
    	    		&(W,g)&\longmapsto&W\sbullet g:=Wg
    	    	\end{array}
    	    \end{equation}
    	is a right action of $T$ on $G(d,n)$.
    \end{definition}
    
    Since the Pl\"ucker embedding\footnote{By definition, the ambient projective space has dimension ${n \choose d}-1=\frac{n!}{d!(n-d)!}-1$.}
    
        \begin{equation*}
        	\psi:G(d,n)\longhookrightarrow\mathds{P}_{\mathds{C}}^{{n\choose d}-1}
        \end{equation*} 
    is a closed immersion, we can define an action of $T$ on $\mathds{P}_{\mathds{C}}^{{n\choose d}-1}$ given by
        \begin{equation}\label{tactionpro}
        	\begin{array}{cccc}
        		\sbullet:&\mathds{P}_{\mathds{C}}^{{n\choose d}-1}\times T&\longrightarrow&\mathds{P}_{\mathds{C}}^{{n\choose d}-1}\\
        		&\left((p_I)_{I\in{[n]\choose d}},g\right)&\longmapsto&(p_I)_{I\in{[n]\choose d}}\sbullet g :=(p_It_I)_{I\in{[n]\choose d}}
        	\end{array},
        \end{equation}
    where $t_I$ denotes the product of the elements on the main diagonal of the columns of $g$ indexed by $I$, and we have the $T$-equivariant diagram
        \begin{equation}\label{grassprodiag}
        	\begin{tikzcd}
        		G(d,n)\times T  \arrow{r}{\sbullet} \arrow{d}[swap]{\psi\times\id}
        		&G(d,n) \arrow[hook]{d}{\psi}\\
        		\mathds{P}_{\mathds{C}}^{{n\choose d}-1}\times T \arrow{r}{\sbullet} &\mathds{P}_{\mathds{C}}^{{n\choose d}-1} 
        	\end{tikzcd}
        \end{equation}
    
    We can define a left action of the symmetric group $\mathfrak{S}_n$ on $G(d,n)$ given by
        \begin{equation}\label{snactiongrass}
        	\begin{array}{cccc}
        		\sbullet:&\mathfrak{S}_n\times G(d,n)&\longrightarrow&G(d,n)\\
        		&(\sigma,W)&\longmapsto&\sigma\sbullet W:=WP_{\sigma}
        	\end{array},
        \end{equation}
    where $P_\sigma$ is the permutation matrix associated to $\sigma$; i.e., if $P_{\sigma}=(\rho_{ij})$, then $\rho_{ij}=1$ if $j=\sigma(i)$ and $\rho_{ij}=0$ otherwise. Likewise, we have a left action of $\mathfrak{S}_n$ on $\mathds{P}^{{n\choose d}-1}_{\mathds{C}}$ given by
        \begin{equation}\label{snactionproj}
        	\begin{array}{cccc}
        		\sbullet:&\mathfrak{S}_n\times\mathds{P}^{{n\choose d}-1}_{\mathds{C}}&\longrightarrow&\mathds{P}^{{n\choose d}-1}_{\mathds{C}}\\
        		&\left(\sigma,(p_I)_{I\in{[n]\choose d}}\right)&\longmapsto&\sigma\sbullet(p_I)_{I\in{[n]\choose d}}:=\left(\varepsilon_{\sigma,I}p_{\sigma^{-1}(I)}\right)_{I\in{[n]\choose d}}
        	\end{array},
        \end{equation}
    where
        \begin{equation*}
        	\varepsilon_{\sigma,I}=\begin{cases*}
        		\ \ 1, & if the number of inversions required to rearrange\\ 
        		&the set $\sigma(I)$ in increasing order is even\\
        		-1, & otherwise
        	\end{cases*}.
        \end{equation*}
        
    We proceed to show that \eqref{snactionproj} satisfies the identity and compatibility axioms of an action:
    	\begin{enumerate}
    		\item $\sigma\sbullet\left(\tau\sbullet(p_I)_{I\in{[n]\choose d}}\right)=(\sigma\tau)\sbullet(p_I)_{I\in{[n]\choose d}}$, $\forall\sigma,\tau\in\mathfrak{S}_n$, $\forall(p_I)_{I\in{[n]\choose d}}\in\mathds{P}^{{n\choose d}-1}_{\mathds{C}}$.
    		    \begin{eqnarray*}
    		    	\sigma\sbullet\left(\tau\sbullet(p_I)_{I\in{[n]\choose d}}\right)&=&\sigma\sbullet(\varepsilon_{\tau,I}p_{\tau^{-1}(I)})_{I\in{[n]\choose d}} \\
    		    	&=&\left(\varepsilon_{\sigma,I}\varepsilon_{\tau,I}p_{\tau^{-1}\big(\sigma^{-1}(I)\big)}\right)_{I\in{[n]\choose d}}\\
    		    	&=&\left(\varepsilon_{\sigma\tau,I}p_{(\sigma\tau)^{-1}(I)}\right)_{I\in{[n]\choose d}}\\
    		    	&=&(\sigma\tau)\sbullet(p_I)_{I\in{[n]\choose d}}.
    		    \end{eqnarray*}
    		\item $\id\sbullet(p_I)_{I\in{[n]\choose d}}=(p_I)_{I\in{[n]\choose d}}$, $\forall(p_I)_{I\in{[n]\choose d}}\in\mathds{P}^{{n\choose d}-1}_{\mathds{C}}$.
    		    \begin{equation*}
    		    	\id\sbullet(p_I)_{I\in{[n]\choose d}}=\left(\varepsilon_{\id,I}p_{\id^{-1}(I)}\right)_{I\in{[n]\choose d}}=\left(p_{\id(I)}\right)_{I\in{[n]\choose d}}=(p_I)_{I\in{[n]\choose d}}.
    		    \end{equation*}
    	\end{enumerate}
    
    \begin{proposition}\label{snequivprop}
    	The Plücker embedding $\psi$ is $\mathfrak{S}_n$-equivariant, i.e., the following diagram
    	    \begin{equation}\label{snequivdiag}
    	    	\begin{tikzcd}
    	    		\mathfrak{S}_n\times G(d,n)\arrow{r}{\sbullet}\arrow{d}[swap]{\operatorname{id}\times\,\psi}
    	    		&G(d,n) \arrow[hook]{d}{\psi}\\
    	    		\mathfrak{S}_n\times\mathds{P}_{\mathds{C}}^{{n\choose d}-1}\arrow{r}{\sbullet} &\mathds{P}_{\mathds{C}}^{{n\choose d}-1} 
    	    	\end{tikzcd}
    	    \end{equation}
    	commutes.
    \end{proposition}
    
    \begin{proof}
    	Let $\sigma\in \mathfrak{S}_n$ and $W=(w_{ij})\in G(d,n)$ arbitrary. Then
    	    \begin{equation*}
    	    	\psi(\sigma\sbullet W)=\psi(WP_{\sigma})=\psi\Big(\!\!\left(w_{i\sigma^{-1}(j)}\right)\!\!\Big)=\left(\varepsilon_{\sigma,I}p_{\sigma^{-1}(I)}\right)_{I\in{[n]\choose d}}=\sigma\sbullet\psi(W).
    	    \end{equation*}
    \end{proof}
    
    Although the following result was not stated explicitly by Elizondo et al., it can be deduced from their arguments in \cite[\S2.2, p. 5]{EFG25}. We will use it in the proof of Lemma \ref{fmaffine}, where we extend it to the analogous construction of the thin Schubert cells in the setting of complex flag varieties.
    
    \begin{proposition}\label{gmaffine}
    	Let $M=([n],\mathcal{B})\in\EuScript{M}^d_n$ be a matroid. The \textbf{thin Schubert cell}
    	    \begin{equation}\label{gm}
    	    	G_M:=\left\{(p_I)_{I\in{[n]\choose d}}\in G(d,n)\mid p_I\neq0\Leftrightarrow I\in\mathcal{B}\right\},
    	    \end{equation}
    	where $(p_I)_{I\in{[n]\choose d}}$ are the Pl\"ucker coordinates is affine and a locally closed set in $G(d,n)$.
    \end{proposition}
    
    If $M$ is a matroid of rank $d$ on $n$ elements whose bases satisfy the exchange property, then $\sigma(M)$ also satisfies it for every $\sigma\in\mathfrak{S}_n$. Formally, we have the following:
    
    \begin{remark}
    	Let $M=([n],\mathcal{B})\in\EuScript{M}^d_n$ and $\sigma\in\mathfrak{S}_n$; then $\sigma(M):=\big([n], \sigma(\mathcal{B})\big)\in\EuScript{M}^d_n$, where $\sigma(\mathcal{B}):=\{\sigma(I)\mid I\in\mathcal{B}\}$.
    \end{remark}
    
    \begin{proof}
    	Indeed, let $\sigma\in \mathfrak{S}_n$ and $I,J\in\mathcal{B}$; then $\#\sigma(I)=\#I$ since $\sigma$ is bijective. On the other hand, we know that for all $i\in I\setminus J$ exists $j\in J\setminus I$ such that $(I\setminus\{i\})\cup\{j\}\in\mathcal{B}$, therefore, $\sigma(i)\in\sigma(I)\setminus\sigma(J)$ and $\sigma(j)\in\sigma(J)\setminus\sigma(I)$. Consequently, $\sigma\big((I\setminus\{i\})\cup\{j\}\big)\in\sigma(\mathcal{B})$ and hence $\big(\sigma(I)\setminus\{\sigma(i)\}\big)\cup\{\sigma(j)\}\in\sigma(\mathcal{B})$.
    \end{proof}
    
    \begin{proposition}\label{gmisoprop}
    	Let $\sigma\in\mathfrak{S}_n$ be an arbitrary permutation. Then $\sigma(G_M):=G_{\sigma(M)}$ is a thin Schubert cell if and only if $G_M$ is a thin Schubert cell. Moreover, there exists an isomorphism $G_M\cong G_{\sigma(M)}$.
    \end{proposition}
    
    \begin{proof}
    	Note that
    	    \begin{eqnarray*}
    	    	\sigma\sbullet(p_I)_{I\in{[n]\choose d}}\in G_{\sigma(M)}&\Longleftrightarrow&\left(\varepsilon_{\sigma,I}p_{\sigma^{-1}(I)}\right)_{I\in{[n]\choose d}}\in G_{\sigma(M)}\\
    	    	&\Longleftrightarrow& \left(p_{\sigma^{-1}(I)}\neq0\Leftrightarrow I\in\sigma(\mathcal{B})\right)\\
    	    	&\Longleftrightarrow&\left(p_{\sigma^{-1}(I)}\neq0\Leftrightarrow\sigma^{-1}(I)\in\mathcal{B}\right)\\
    	    	&\Longleftrightarrow&\left(p_I\neq0\Leftrightarrow I\in\mathcal{B}\right)\\
    	    	&\Longleftrightarrow&(p_I)_{I\in{[n]\choose d}}\in G_M.
    	    \end{eqnarray*}
    	The morphism
    	    \begin{equation}\label{gmisomorphism}
    	    	\begin{array}{cccc}
    	    		\varphi_{\sigma}:&G_M&\longrightarrow&G_{\sigma(M)}\\
    	    		&(p_I)_{I\in{[n]\choose d}}&\longmapsto&\varphi_{\sigma}\left((p_I)_{I\in{[n]\choose d}}\right):=\sigma\sbullet(p_I)_{I\in{[n]\choose d}}
    	    	\end{array}
    	    \end{equation}
    	is the desired isomorphism, whose inverse is given by $\varphi_{\sigma}^{-1}:=\varphi_{\sigma^{-1}}$.
    \end{proof}
    
\section{Thin Schubert cells and matroidal decompositions on flag varieties}\label{sec3}
    In this section we introduce the analogue of thin Schubert cells in a complex flag variety and establish a decomposition of this variety in terms of these cells. To this end, we consider the flag variety
        \begin{equation*}
        	\mathds{F}_{d_1<\cdots<d_k}(\mathds{C}^n):=\{W_1\subset\cdots\subset W_k\subset\mathds{C}^n\mid \mathrm{dim}\,W_i=d_i,\ \forall i=1,\ldots,k\}\subset\prod_{i=1}^kG(d_i,n).
        \end{equation*}
        
    The actions of the torus $T$ on Grassmannians and projective spaces, as introduced in \S\ref{subsec2.1} and defined in \eqref{tactiongrass} and \eqref{tactionpro}, respectively, extend naturally to Cartesian products of such varieties. Consequently, the algebro-geometric structure induced by the $T$-actions is preserved under these extensions. We now state this result formally.
    
    \begin{proposition}\label{tequivariance}
    	By convention, let $(\mathbf{p_I})_{\mathbf{I}\in{[n]\choose d_1,\ldots,d_k}} :=\left((p_{I_1})_{I_1\in{[n]\choose d_1}},\ldots,(p_{I_k})_{I_k\in{[n]\choose d_k}}\right)$ and consider the actions of $T$ given by
    	    \begin{equation}
    	    	\begin{array}{cccc}\label{tactioncargrass}
    	    		\sbullet:&\displaystyle\prod_{i=1}^k G(d_i,n)\times T&\longrightarrow&\displaystyle\prod_{i=1}^kG(d_i,n)\\
    	    		&(\mathbf{W},g)&\longmapsto&\mathbf{W}\sbullet g:=(W_1g,\ldots,W_kg)
    	    	\end{array}
    	    \end{equation}
    	and
    	    \begin{equation}
    	    	\begin{array}{cccc}\label{tactioncarpro}
    	    		\sbullet:&\displaystyle\prod_{i=1}^k\mathds{P}_{\mathds{C}}^{{n\choose d_i}-1}\times T&\longrightarrow&\displaystyle\prod_{i=1}^k\mathds{P}_{\mathds{C}}^{{n\choose d_i}-1}\\
    	    		&\left((\mathbf{p_I})_{\mathbf{I}\in{[n]\choose d_1,\ldots,d_k}},g\right)&\longmapsto&(\mathbf{p_I})_{\mathbf{I}\in{[n]\choose d_1,\ldots,d_k}}\sbullet g :=\left((p_{I_1}t_{I_1})_{I_1\in{[n]\choose d_1}},\ldots,(p_{I_k}t_{I_k})_{I_k\in{[n]\choose d_k}}\right)
    	    	\end{array}.
    	    \end{equation}
    	Then the map
    	    \begin{equation}\label{genpluck}
    	    	\begin{array}{cccc}
    	    		\Psi:&\displaystyle\prod_{i=1}^kG(d_i,n)&\longhookrightarrow&\displaystyle\prod_{i=1}^k\mathds{P}_{\mathds{C}}^{{n\choose d_i}-1}\\
    	    		&\mathbf{W}&\longmapsto&\big(\psi_1(W_1),\ldots,\psi_k(W_k)\big)
    	    	\end{array}
    	    \end{equation}
    	where $\Psi:=\psi_1\times\cdots\times\psi_k$ is $T$-equivariant; i.e., the diagram
    	    \begin{equation}\label{cargrassprodiag}
    	    	\begin{tikzcd}
    	    		\displaystyle\prod_{i=1}^kG(d_i,n)\times T  \arrow{r}{\sbullet} \arrow[hook]{d}[swap]{\Psi\times\operatorname{id}\,}
    	    		&\displaystyle\prod_{i=1}^kG(d_i,n) \arrow[hook]{d}{\Psi}\\
    	    		\displaystyle\prod_{i=1}^k\mathds{P}_{\mathds{C}}^{{n\choose d_i}-1}\times T \arrow{r}{\sbullet} &\displaystyle\prod_{i=1}^k\mathds{P}_{\mathds{C}}^{{n\choose d_i}-1} 
    	    	\end{tikzcd}
    	    \end{equation}
    	commutes.
    \end{proposition}
    
    We will also need a lemma that will play an important role in \S\ref{sec4}, which we state now.
     
    \begin{lemma}\label{fmaffine}
    	Let $\mathbf{M}=(M_1,\ldots,M_k)\in\EuScript{M}^{d_1,\ldots,d_k}_n$ be a plurimatroid. Then the \textbf{thin Schubert cell} 
    	    \begin{align}
    	    	\begin{split}\label{celshuflag}
    	    		\mathds{F}_{\mathbf{M}}:=&\left(\prod_{i=1}^{k}G_{M_i}\right)\cap\mathds{F}_{d_1<\cdots<d_k}(\mathds{C}^n)\\
    	    		=&\left\{\big((p_{I_1}),\ldots,(p_{I_k})\big)_{(I_1,\ldots,I_k)\in{[n]\choose d_1,\ldots,d_k}}\in\mathds{F}_{d_1<\cdots<d_k}(\mathds{C}^n)\mid p_{I_j}\neq0\Leftrightarrow I_j\in B_{M_j},\ j=1,\ldots,k\right\}
    	    	\end{split}
    	    \end{align}
    	is affine and a locally closed set in $\mathds{F}_{d_1<\cdots<d_k}(\mathds{C})$.
    \end{lemma}
    
    \begin{proof}
    	First, note that $\displaystyle\prod_{i=1}^{k}G_{M_i}$ is affine and a locally closed in $\displaystyle\prod_{i=1}^{k}G(d_i,n)$, since each $G_{M_i}$ is affine and a locally closed set in $G(d_i,n)$ by Proposition \ref{gmaffine}. Hence, so is $\mathds{F}_{\mathbf{M}}$.
    \end{proof}
    
    \begin{remark}
    	It is necessary to emphasize here \textbf{our definition of thin Schubert cell} introduced in Lemma \ref{fmaffine} \textbf{should not be confused} with the one established by Gelfand and Serganova in \cite[\S\S4.1 and 6.2, pp. 152 and 154]{GS87}.
    \end{remark}
    
    We are now in a position to establish the desired decomposition. This will enable us to prove Lemma \ref{lemclos}, which will be fundamental in proving the first part of Theorem \ref{subvar-ugq}.
    
    \begin{proposition}\label{fdecomp}
    	We have the following decomposition:
    	    \begin{equation}\label{fmdecomp}
    	    	\mathds{F}_{d_1<\cdots<d_k}(\mathds{C}^n)=\bigsqcup_{\mathbf{M}\in\EuScript{M}^{d_1,\ldots,d_k}_n}\mathds{F}_{\mathbf{M}}.
    	    \end{equation}
    \end{proposition}
    \begin{proof}
    	$(\subset)$ We proceed by a direct calculation:
    	    \begin{align*}
    	    	\mathds{F}_{d_1<\cdots<d_k}(\mathds{C}^n)&\subset\left(\prod_{i=1}^kG(d_i,n)\right)\cap\mathds{F}_{d_1<\cdots<d_k}(\mathds{C}^n)\\
    	    	&=\left(\prod_{i=1}^k\left(\bigsqcup_{M_i\in\EuScript{M}^{d_i}_n}G_{M_i}\right)\right)\cap\mathds{F}_{d_1<\cdots<d_k}(\mathds{C}^n) \tag*{(by \cite[Prop. 2.4, p. 5]{EFG25})}\\
    	    	&=\left(\bigsqcup_{\mathbf{M}\in\EuScript{M}^{d_1,\ldots,d_k}_n}\left(\prod_{i=1}^kG_{M_i}\right)\right)\cap\mathds{F}_{d_1<\cdots<d_k}(\mathds{C}^n)\\
    	    	&=\bigsqcup_{\mathbf{M}\in\EuScript{M}^{d_1,\ldots,d_k}_n}\left(\left(\prod_{i=1}^kG_{M_i}\right)\cap\mathds{F}_{d_1<\cdots<d_k}(\mathds{C}^n)\right)\\
    	    	&=\bigsqcup_{\mathbf{M}\in\EuScript{M}^{d_1,\ldots,d_k}_n}\mathds{F}_{\mathbf{M}}, \tag*{(by definition of $\mathds{F}_{\mathbf{M}}$)}
    	    \end{align*}
    	$(\supset)$ It follows quickly since $\mathds{F}_{\mathbf{M}}\subset\mathds{F}_{d_1<\cdots<d_k}(\mathds{C}^n)$ for each plurimatroid $\mathbf{M}\in\EuScript{M}^{d_1,\ldots,d_k}_n$. 
    	
    	\noindent Finally, \eqref{fmdecomp} is a decomposition by Lemma \ref{fmaffine}.        
    \end{proof}
    
    Let $\mathbf{M}\in\EuScript{M}^{d_1,\ldots,d_k}_n$ be arbitrary. If $g\in T$ and $\mathbf{\Lambda}\in\mathds{F}_{\mathbf{M}}$, then $\mathbf{\Lambda}\sbullet g\in\mathds{F}_{\mathbf{M}}$ and consequently the thin Schubert cell $\mathds{F}_{\mathbf{M}}$ is $T$-invariant. Furthermore, any two elements in $\mathds{F}_{\mathbf{M}}$ have the same isotropy group, which will be denoted by $\textrm{Stab}_T(\mathds{F}_{\mathbf{M}})$ and is a normal subgroup of $T$. Hence, the quotient
        \begin{equation}\label{subtorflag}
        	T_{\mathbf{M}}:=\frac{T}{\mathrm{Stab}_T(\mathds{F}_{\mathbf{M}})}
        \end{equation}
    is an abelian group that can be identified with a \textit{subtorus} of $T$.
    
    The action of $T$ on $\mathds{F}_{\mathbf{M}}$ induces a set-theoretic free action of $T_{\mathbf{M}}$ on $\mathds{F}_{\mathbf{M}}$ given by
        \begin{equation*}\label{tmaction}
        	\begin{array}{cccc}
        		\sbullet:&\mathds{F}_{\mathbf{M}}\times T_{\mathbf{M}}&\longrightarrow&\mathds{F}_{\mathbf{M}}\\
        		&(\mathbf{\Lambda},[g])&\longmapsto&\mathbf{\Lambda}\,\!\sbullet\,\![g]:=\mathbf{\Lambda}\sbullet g
        	\end{array}.
        \end{equation*}
    
    For each $\mathbf{\Lambda}\in\mathds{F}_{\mathbf{M}}$ we have $\mathbf{\Lambda}\sbullet T_{\mathbf{M}}=\mathbf{\Lambda}\sbullet T$. Likewise, the map
        \begin{equation*}
        	\begin{array}{cccc}
        		\theta:&T_{\mathbf{M}}&\longrightarrow&\mathbf{\Lambda}\sbullet T_{\mathbf{M}}\\
        		&[g]&\longmapsto&\mathbf{\Lambda}\sbullet\,\![g]
        	\end{array}
        \end{equation*}
    is a bijection between $T_{\mathbf{M}}$ and the orbit $\mathbf{\Lambda}\sbullet T_{\mathbf{M}}$.  
    
    Let $\nicefrac{\mathds{F}_{\mathbf{M}}}{T_{\mathbf{M}}}$ be the set of orbits of $\mathds{F}_{\mathbf{M}}$ under the action of $T_{\mathbf{M}}$ and consider the topological space $(\mathds{F}_{\mathbf{M}},\tau_{\text{\tiny Zar}})$; then the map 
        \begin{equation}\label{quomap}
        	\begin{array}{cccc}
        		\pi:&\mathds{F}_{\mathbf{M}}&\longrightarrow&\dfrac{\mathds{F}_{\mathbf{M}}}{T_{\mathbf{M}}}\\
        		&\mathbf{\Lambda}&\longrightarrow&\mathbf{\Lambda}\sbullet T_{\mathbf{M}}
        	\end{array}
        \end{equation}
    is surjective and induces a quotient topology on $\nicefrac{\mathds{F}_{\mathbf{M}}}{T_{\mathbf{M}}}$, making it an orbit space. Thus, we have a \emph{fiber bundle}
        \begin{equation}\label{fibflag}
        	T_{\mathbf{M}}\longrightarrow\mathds{F}_{\mathbf{M}}\stackrel{\pi}{\longrightarrow}\frac{\mathds{F}_{\mathbf{M}}}{T_{\mathbf{M}}},
        \end{equation}
    
    These bundles are topologically trivial $\left(\mathds{F}_{\mathbf{M}}\cong\frac{\mathds{F}_{\mathbf{M}}}{T_{\mathbf{M}}}\times T_{\mathbf{M}}\right)$: it is possible to construct a representation $\mathbf{\Lambda}_0$ of an arbitrary element $W_{\sbullet}:=W_1\subset\cdots\subset W_k\in\mathds{F}_{\mathbf{M}}$; consequently, there is a section
        \begin{equation*}
        	\begin{tikzcd}
        		T_{\mathbf{M}}\arrow{r}&\mathds{F}_{\mathbf{M}}\arrow{r}{\pi}&\arrow[bend left]{l}{s}\dfrac{\mathds{F}_{\mathbf{M}}}{T_{\mathbf{M}}}\
        	\end{tikzcd}
        \end{equation*}
    given by
        \begin{align*}
        	\begin{array}{cccc}
        		s:&\dfrac{\mathds{F}_{\mathbf{M}}}{T_{\mathbf{M}}}&\longrightarrow&\mathds{F}_{\mathbf{M}}\\
        		&\mathbf{\Lambda}\sbullet T_{\mathbf{M}}&\longmapsto&\mathbf{\Lambda_0}
        	\end{array}
        \end{align*}   
    together with morphisms    
        \begin{equation*}
        	\begin{array}{cccc}
        		\alpha:&\mathds{F}_{\mathbf{M}}&\longrightarrow&\dfrac{\mathds{F}_{\mathbf{M}}}{T_{\mathbf{M}}}\times T_{\mathbf{M}}\\
        		&\mathbf{\Lambda}&\longmapsto&\big(\pi(\mathbf{\Lambda}),t\big)
        	\end{array}
        \end{equation*}
    where $t\in T_{\mathbf{M}}$ is such that $\mathbf{\Lambda}=\mathbf{\Lambda_0}\sbullet t$ and
        \begin{equation*}
        	\begin{array}{cccc}
        		\beta:&\dfrac{\mathds{F}_{\mathbf{M}}}{T_{\mathbf{M}}}\times T_{\mathbf{M}}&\longrightarrow&\mathds{F}_{\mathbf{M}}\\
        		&(\mathbf{\Lambda}\sbullet T_{\mathbf{M}},t)&\longmapsto&s(\mathbf{\Lambda}\sbullet T_{\mathbf{M}})\sbullet t
        	\end{array}
        \end{equation*} 
    that are inverses of each other. Let us now consider an example illustrating the preceding discussion.
    
    \begin{example}
    	Consider the flag variety $\mathds{F}_{2<3}(\mathds{C}^4)$, $\mathcal{B}=\big\{\{1,2\},\{1,3\},\{1,4\},\{2,3\},\{2,4\},\{3,4\}\big\}$ and $\mathcal{B}'=\big\{\{1,2,3\},\{1,2,4\},\{1,3,4\},\{2,3,4\}\big\}$ sets of bases for the corresponding matroids $M=([4],\mathcal{B})\in\EuScript{M}^{2}_4$ and $M'=([4],\mathcal{B}')\in\EuScript{M}^{3}_4$, respectively. Let $T<\textrm{GL}_4(\mathds{C})$ be the maximal torus and $\mathbf{M}=(M,M')\in\EuScript{M}^{2,3}_4$. Then
    	    \begin{equation*}%\scriptsize
    	    	\mathds{F}_{\mathbf{M}}=\left\{\begin{pmatrix*}
    	    		1&0&x_{13}&x_{14}\\
    	    		0&1&x_{23}&x_{24}
    	    	\end{pmatrix*}\subset\begin{pmatrix*}
    	    		1&0&0&x_{14}-x_{13}y_{34}\\ %\vspace*{0cm}
    	    		0&1&0&x_{24}-x_{23}y_{34}\\ 
    	    		0&0&1&y_{34}
    	    	\end{pmatrix*}:\begin{array}{c}
    	    		x_{1j},x_{2j},y_{34}\in\mathds{C}^*\\
    	    		x_{13}x_{24}-x_{14}x_{23}\neq0\\
    	    		x_{i4}-x_{i3}y_{34}\neq0
    	    	\end{array}\hspace*{-0.15cm}       	  
    	    	\right\}.
    	    \end{equation*}
    	
    	The representative matrix of an element $W_{\sbullet}\in\mathds{F}_{\mathbf{M}}$ is
    	    \begin{equation}\label{matrep3}
    	    	\mathbf{\Lambda}=\begin{pmatrix*}
    	    		1&0&x_{13}&x_{14}\\
    	    		0&1&x_{23}&x_{24}\\
    	    		0&0&1&y_{34}
    	    	\end{pmatrix*},
    	    \end{equation}
    	the isotropy group is $\textrm{Stab}_{T}(\mathds{F}_{\mathbf{M}})=\{t\in T\mid t_1=t_2=t_3=t_4\}\cong\mathds{C}^*$, which implies that $T_{\mathbf{M}}\cong(\mathds{C}^*)^3$. Since $\dim\mathds{F}_{\mathbf{M}}=5>3=\dim T_{\mathbf{M}}$, it is necessary to fix two parameters, say $\lambda_1=\dfrac{x_{14}}{x_{13}y_{34}}\rule{0pt}{1.5em}$ and $\lambda_2=\dfrac{x_{24}}{x_{23}y_{34}}$. Thus, we can rewrite the matrix \eqref{matrep3} as
    	    \begin{equation*}
    	    	\mathbf{\Lambda}=\begin{pmatrix*}
    	    		1&0&x_{13}&\lambda_1 x_{13}y_{34}\\
    	    		0&1&x_{23}&\lambda_2 x_{23}y_{34}\\
    	    		0&0&1&y_{34}
    	    	\end{pmatrix*},
    	    \end{equation*}
    	where setting the variable entries to one we obtain
    	    \begin{equation*}
    	    	\mathbf{\Lambda_0}=\begin{pmatrix*}
    	    		1&0&1&\lambda_1\\
    	    		0&1&1&\lambda_2\\
    	    		0&0&1&1
    	    	\end{pmatrix*}.
    	    \end{equation*}
    	
    	Now, consider the matrices
    	    \begin{equation*}
    	    	\begin{aligned}
    	    		s=\begin{pmatrix*}\vspace*{0.1cm}
    	    			1&0&0\\ \vspace*{0.2cm}
    	    			0&\dfrac{x_{23}}{x_{13}}&0\\ \vspace*{0.1cm}
    	    			0&0&\dfrac{1}{x_{13}} 
    	    		\end{pmatrix*}\in\textrm{GL}_3(\mathds{C})
    	    	\end{aligned}\quad\text{and}\quad
    	    	\begin{aligned}
    	    		t=\begin{pmatrix*}\vspace*{0.1cm}
    	    			1&0&0&0\\ \vspace*{0.2cm}
    	    			0&\dfrac{x_{13}}{x_{23}}&0&0\\
    	    			0&0&x_{13}&0\\
    	    			0&0&0&x_{13}y_{34} 
    	    		\end{pmatrix*}\in T_{\mathbf{M}}
    	    	\end{aligned}
    	    \end{equation*}
    	we have
    	    \begin{align*}
    	    	s\,\mathbf{\Lambda_0}\, t&=\begin{pmatrix*}\vspace*{0.1cm}
    	    		1&0&0\\ \vspace*{0.2cm}
    	    		0&\dfrac{x_{23}}{x_{13}}&0\\ \vspace*{0.1cm}
    	    		0&0&\dfrac{1}{x_{13}} 
    	    	\end{pmatrix*}\begin{pmatrix*}
    	    		1&0&1&\lambda_1\\
    	    		0&1&1&\lambda_2\\
    	    		0&0&1&1
    	    	\end{pmatrix*}\begin{pmatrix*}\vspace*{0.1cm}
    	    		1&0&0&0\\ \vspace*{0.2cm}
    	    		0&\dfrac{x_{13}}{x_{23}}&0&0\\
    	    		0&0&x_{13}&0\\
    	    		0&0&0&x_{13}y_{34} 
    	    	\end{pmatrix*}\\
    	    	&=\begin{pmatrix*}\vspace*{0.1cm}
    	    		1&0&0\\ \vspace*{0.2cm}
    	    		0&\dfrac{x_{23}}{x_{13}}&0\\ \vspace*{0.075cm}
    	    		0&0&\dfrac{1}{x_{13}} 
    	    	\end{pmatrix*}\begin{pmatrix*}\vspace*{0.1cm}
    	    		1&0&x_{13}&\lambda_1 x_{13}y_{34}\\ \vspace*{0.1cm}
    	    		0&\dfrac{x_{13}}{x_{23}}&x_{13}&\lambda_2 x_{13}y_{34}\\
    	    		0&0&x_{13}&x_{13}y_{34}
    	    	\end{pmatrix*}\\
    	    	&=\begin{pmatrix*}
    	    		1&0&x_{13}&\lambda_1 x_{13}y_{34}\\
    	    		0&1&x_{23}&\lambda_2 x_{23}y_{34}\\
    	    		0&0&1&y_{34}
    	    	\end{pmatrix*}\\
    	    	&=\mathbf{\Lambda}.
    	    \end{align*}
    	We conclude that $\mathbf{\Lambda_0}\sbullet t=\mathbf{\Lambda}$ (as flags) and consequently
    	    \begin{equation*}
    	    	\frac{\mathds{F}_{\mathbf{M}}}{T_{\mathbf{M}}}\cong(\mathds{C}^*)^2\setminus\mathcal{V}\big((\lambda_2-\lambda_1){\cdot}(\lambda_1-1){\cdot}(\lambda_2-1)\big).
    	    \end{equation*}
    \end{example}
    
    The actions of the symmetric group $\mathfrak{S}_n$ on Grassmannians and projective spaces, as introduced in \S\ref{subsec2.1} and defined in \eqref{snactiongrass} and \eqref{snactionproj}, respectively, extend naturally to Cartesian products of such varieties. We now state this result formally.
    
    \begin{proposition}\label{snequivariance}
    	Consider the actions of $\mathfrak{S}_n$ given by
    	    \begin{equation}\label{snactioncargrass}
    	    	\begin{array}{cccc}
    	    		\sbullet:&\mathfrak{S}_n\times\displaystyle\prod_{i=1}^{k}G(d_i,n)&\longrightarrow&\displaystyle\prod_{i=1}^{k}G(d_i,n)\\
    	    		&(\sigma,\mathbf{W})&\longmapsto&\sigma\sbullet\mathbf{W}:=(\sigma\sbullet W_1,\ldots,\sigma\sbullet W_k)
    	    	\end{array}
    	    \end{equation}      
    	and
    	    \begin{equation}\label{snactioncarproj}
    	    	\begin{array}{cccc}
    	    		\sbullet:&\mathfrak{S}_n\times\displaystyle\prod_{i=1}^k\mathds{P}^{{n\choose d_i}-1}_{\mathds{C}}&\longrightarrow&\displaystyle\prod_{i=1}^k\mathds{P}^{{n\choose d_i}-1}_{\mathds{C}}\\
    	    		&\left(\sigma,(\mathbf{p_I})_{\mathbf{I}\in{[n]\choose d_1,\ldots,d_k}}\right)&\longmapsto&\sigma\sbullet(\mathbf{p_I})_{\mathbf{I}\in{[n]\choose d_1,\ldots,d_k}}:=\left(\sigma\sbullet(p_{I_1})_{I_1\in{[n]\choose d_1}},\ldots,\sigma\sbullet(p_{I_k})_{I_k\in{[n]\choose d_k}}\right)
    	    	\end{array}.
    	    \end{equation}
    	
    	Then, the map $\Psi$ introduced in Proposition \ref{tequivariance} and defined in \eqref{genpluck}, is $\mathfrak{S}_n$-equivariant; i.e., the next diagram commutes
    	    \begin{equation}\label{snequivariancediag}
    	    	\begin{tikzcd}
    	    		\mathfrak{S}_n\times\displaystyle\prod_{i=1}^kG(d_i,n)\arrow{r}{\sbullet} \arrow[hook]{d}[swap]{\operatorname{id}\times\Psi}
    	    		&\displaystyle\prod_{i=1}^kG(d_i,n) \arrow[hook]{d}{\Psi}\\
    	    		\mathfrak{S}_n\times\displaystyle\prod_{i=1}^k\mathds{P}_{\mathds{C}}^{{n\choose d_i}-1}\arrow{r}{\sbullet} &\displaystyle\prod_{i=1}^k\mathds{P}_{\mathds{C}}^{{n\choose d_i}-1} 
    	    	\end{tikzcd}
    	    \end{equation}
    \end{proposition}
    
    \begin{remark}
    	The flag variety is $\mathfrak{S}_n$-invariant; that is, if $\sigma\in\mathfrak{S}_n$ and $W_{\sbullet}\in\mathds{F}_{d_1<\cdots<d_k}(\mathds{C}^n)$ be arbitrary, then $\sigma\sbullet W_{\sbullet}:=\sigma\sbullet W_1\subset\cdots\subset\sigma\sbullet W_k\in\mathds{F}_{d_1<\cdots<d_k}(\mathds{C}^n)$.
    \end{remark}
    
    \begin{remark}
    	Let $\sigma\in\mathfrak{S}_n$ and $\mathbf{M}\in\mathscr{M}^{d_1,\ldots,d_k}_n$. Then $\sigma(\mathbf{M}):=\sigma(M_1)\times\cdots\times\sigma(M_k)\in\mathscr{M}^{d_1,\ldots,d_k}_n$.
    \end{remark}
    
    \begin{proposition}\label{fmisoprop}
    	Let $\sigma\in\mathfrak{S}_n$ be an arbitrary permutation. Then $\sigma(\mathds{F}_{\mathbf{M}}):=\mathds{F}_{\sigma(\mathbf{M})}$ is a thin Schubert cell if and only if $\mathds{F}_{\mathbf{M}}$ is a thin Schubert cell. Moreover, there exists an isomorphism $\mathds{F}_{\mathbf{M}}\cong\mathds{F}_{\sigma(\mathbf{M})}$.
    \end{proposition}
    
    \begin{proof}
    	The first part follows immediately from Proposition \ref{gmisoprop}. On the other hand, the morphism
    	    \begin{equation}\label{fmisomorphism}
    	    	\begin{array}{cccc}
    	    		\varPhi_{\sigma}:&\mathds{F}_{\mathbf{M}}&\longrightarrow&\mathds{F}_{\sigma(\mathbf{M})}\\
    	    		&\left((\mathbf{p_I})_{\mathbf{I}\in{[n]\choose d_1,\ldots,d_k}}\right)&\longmapsto&\varPhi_{\sigma}\left((\mathbf{p_I})_{\mathbf{I}\in{[n]\choose d_1,\ldots,d_k}}\right):=\sigma\sbullet(\mathbf{p_I})_{\mathbf{I}\in{[n]\choose d_1,\ldots,d_k}}
    	    	\end{array},
    	    \end{equation}
    	where $\varPhi_{\sigma}:=\varphi^{(1)}_{\sigma}\times\cdots\times\varphi^{(k)}_{\sigma}$ naturally denotes the Cartesian product of the isomorphisms $\varphi^{(s)}_{\sigma}:G_{M_s}\longrightarrow G_{\sigma(M_s)}$ \big(see Proposition \ref{gmisoprop}\eqref{gmisomorphism}\big), is the desired isomorphism whose inverse is given by $\varPhi^{-1}_{\sigma}=\varPhi_{\sigma^{-1}}$.
    \end{proof}
    
\section{Algebro-geometric quotients and the homology of $T$-invariant subvarieties}\label{sec4}
    This section constitutes the core of the paper. Here, we prove the existence of universal geometric quotients of thin Schubert cells and $T$-invariant subvarieties, which, in particular, are universal categorical quotients. Furthermore, we provide a characterization of the $T$-invariant subvarieties in terms of the thin Schubert cell with which they intersect densely and the quotient associated with this intersection.
    
    \begin{theorem}\label{fm-ugq}
    	The universal geometric quotient of $\mathds{F}_{\mathbf{M}}$ by $T_{\mathbf{M}}$, denoted by $\EuScript{F}_{\mathbf{M}}$ exists, is affine and the morphism
    	    \begin{align}\label{pimor}
    	    	\begin{array}{cccc}
    	    		\pi:&\mathds{F}_{\mathbf{M}}&\longrightarrow&\EuScript{F}_{\mathbf{M}}\\
    	    		&\mathbf{\Lambda}&\longmapsto&\mathbf{\Lambda}\sbullet T_{\mathbf{M}}
    	    	\end{array}
    	    \end{align}
    	is universally submersive; in particular, $\EuScript{F}_{\mathbf{M}}$ is a universal categorical quotient.
    \end{theorem}
    
    \begin{proof}
    	Since $\mathds{F}_{\mathbf{M}}$ is affine, $T_{\mathbf{M}}$ is a reductive algebraic group acting on it with all its orbits closed, and $\mathds{C}$ is algebraically closed, the Theorem 1.1 in \cite[Chap. 1, \S2, p. 27]{MFK94} and the Amplification 1.3 in \cite[Chap. 1, \S2, p. 30]{MFK94} guarantee the existence of $\EuScript{F}_{\mathbf{M}}$ as a universal geometric quotient and show that the morphism \eqref{pimor} is universally submersive. Moreover, $\EuScript{F}_{\mathbf{M}}$ is a universal categorical quotient and is unique up to isomorphism by Proposition 0.1 in \cite[Chap. 0, \S2, p. 4]{MFK94}.
    \end{proof}
    
    \begin{remark}
    	The closed points of $\EuScript{F}_{\mathbf{M}}$ are in bijective correspondence with the orbits of $T_{\mathbf{M}}$ on $\mathds{F}_{\mathbf{M}}$.\vspace*{1ex}
    \end{remark}	
    
    We need the next lemma to prove two of the main results of this work, namely Theorem \ref{subvar-ugq} and Theorem \ref{mainthm}.
    
    \begin{lemma}\label{lemclos}
    	Let $X\subset\mathds{F}_{d_1<\cdots<d_k}(\mathds{C}^n)$ be a closed set; then 
    	    \begin{equation*}
    	    	X=\displaystyle\bigcup_{\mathbf{M}\in\EuScript{M}^{d_1,\ldots,d_k}_n}\overline{X\cap\mathds{F}_{\mathbf{M}}}.
    	    \end{equation*}
    	If $X$ is also irreducible, then there exists a unique $\mathbf{M}\in\EuScript{M}_n^{d_1,\ldots,d_k}$ such that $X=\overline{X\cap\mathds{F}_{\mathbf{M}}}$.
    \end{lemma}	
    
    \begin{proof}
    	Note that
    	    \begin{align*}
    	    	X &=\overline{X}\\ %\tag*{por ser $Y$ cerrado}\\
    	    	&=\overline{X\cap\mathds{F}_{d_1<\cdots<d_k}(\mathds{C}^n)}\\ %\tag*{porque $Y\subset\mathds{F}_{d_1<\cdots<d_k}(\mathds{C}^n)$}\\
    	    	&=\overline{X\cap\left(\bigsqcup_{\mathbf{M}\in\EuScript{M}^{d_1,\ldots,d_k}_n}\mathds{F}_{\mathbf{M}}\right)} \tag*{(by Proposition \ref{fdecomp})}\\
    	    	&=\overline{\bigsqcup_{\mathbf{M}\in\EuScript{M}^{d_1,\ldots,d_k}_n}X\cap\mathds{F}_{\mathbf{M}}}\\
    	    	&=\bigcup_{\mathbf{M}\in\EuScript{M}^{d_1,\ldots,d_k}_n}\overline{X\cap\mathds{F}_{\mathbf{M}}}.
    	    \end{align*}
    	The second statement is immediate, which completes the proof of the lemma.
    \end{proof}
    
    At this point, we can now prove that the intersection of a $T$-invariant subvariety with a thin Schubert cell, modulo the action of a subtorus, admits a universal geometric quotient. Moreover, we can show that this subvariety is isomorphic to a fiber product of this geometric quotient with the corresponding thin cell.
    
    \begin{theorem}\label{subvar-ugq}
    	Let $X\subset\mathds{F}_{d_1<\cdots<d_k}(\mathds{C}^n)$ be a $T$-invariant subvariety; then there exists a unique plurimatroid $\mathbf{M}\in\EuScript{M}_n^{d_1,\ldots,d_k}$ such that
    	    \begin{enumerate}
    	    	\item The universal geometric quotient $\mathcal{Y}=\frac{X\cap\mathds{F}_{\mathbf{M}}}{T_{\mathbf{M}}}$ exists, is irreducible, and the morphism
    	    	\begin{align}\label{tilpimor}
    	    		\begin{array}{cccc}
    	    			\widetilde{\pi}:&X\cap\mathds{F}_{\mathbf{M}}&\longrightarrow&\mathcal{Y}\\
    	    			&\mathbf{\Lambda}&\longmapsto&\mathbf{\Lambda}\sbullet T_{\mathbf{M}}
    	    		\end{array}
    	    	\end{align}
    	    	is universally submersive; in particular, $\mathcal{Y}$ is a universal categorical quotient.
    	    	\item $X\cong\overline{\mathcal{Y}\times_{\EuScript{F}_{\mathbf{M}}}\mathds{F}_{\mathbf{M}}}$.
    	    \end{enumerate}
    \end{theorem}
    
    \begin{proof}
    	First, since $X$ and $\mathds{F}_{\mathbf{M}}$ are $T$-invariant and affine, so is $X\cap\mathds{F}_{\mathbf{M}}$; second, $X\cap\mathds{F}_{\mathbf{M}}\subset\mathds{F}_{\mathbf{M}}$ implies that all points of $X\cap\mathds{F}_{\mathbf{M}}$ have the same stabilizer in $T$, so $T_{\mathbf{M}}=\faktor{T}{\mathrm{Stab}_T(X\cap\mathds{F}_{\mathbf{M}})}$ acts freely on $X\cap\mathds{F}_{\mathbf{M}}$ with all its orbits closed. Theorem 1.1 in \cite[Chap. 1, \S2, p. 27]{MFK94} together with Amplification 1.3 in \cite[Chap. 1, \S2, p. 30]{MFK94} guarantee the existence of the universal geometric quotient and show that the morphism \eqref{tilpimor} is universally submersive. Moreover, $\mathcal{Y}$ is a universal categorical quotient, and by Proposition 0.1 in \cite[Chap. 0, \S2, p. 4]{MFK94} it is unique up to isomorphism. Since the orbits of $T$ and $T_{\mathbf{M}}$ coincide, we have the fiber bundle
    	    \begin{equation*}
    	    	T_{\mathbf{M}}\longrightarrow X\cap\mathds{F}_{\mathbf{M}}\stackrel{\widetilde{\pi}}{\longrightarrow}\mathcal{Y},
    	    \end{equation*}
    	where $\widetilde{\pi}={\pi\bigr|}_{X\cap\mathds{F}_{\mathbf{M}}}$, with $\pi$ defined as \eqref{pimor}. Furthermore, since $X$ is irreducible, by Lemma \ref{lemclos} there exists a unique $\mathbf{M}\in\EuScript{M}^{d_1,\ldots,d_k}_n$ such that $X=\overline{X\cap\mathds{F}_{\mathbf{M}}}$, it follows that  $X\cap\mathds{F}_{\mathbf{M}}\subset X$ is dense and hence irreducible, since $X$ is. Consequently, the quotient $\mathcal{Y}=\frac{X\cap\mathds{F}_{\mathbf{M}}}{T_{\mathbf{M}}}$ is also irreducible.				
    	
    	Consider the commutative diagram
    	    \begin{equation*}
    	    	\xymatrix{
    	    		X\cap\mathds{F}_{\mathbf{M}}\ar[d]_{\widetilde{\pi}}\,\ar@{^(->}[r]^{\quad i_2} & \mathds{F}_{\mathbf{M}}\ar[d]^{\pi}\\
    	    		\mathcal{Y}\,\ar@{^(->}[r]^{i_1} & \EuScript{F}_{\mathbf{M}}
    	    	}
    	    \end{equation*}
    	where $i_1,i_2$ are the canonical inclusions and $\widetilde{\pi}={\pi\bigr|}_{X\cap\mathds{F}_{\mathbf{M}}}$. Now, since the fiber product $\mathcal{Y}\times_{\EuScript{F}_{\mathbf{M}}}\mathds{F}_{\mathbf{M}}$ is unique up to isomorphism, we will prove that $X\cap\mathds{F}_{\mathbf{M}}$ satisfies the universal property: if $\phi_1:Z\longrightarrow\mathcal{Y}$ and $\phi_2:Z\longrightarrow\mathds{F}_{\mathbf{M}}$ are morphisms such that $\pi\circ\phi_2=i_1\circ\phi_1$, then there exists a unique morphism $u:Z\longrightarrow X\cap\mathds{F}_{\mathbf{M}}$ that makes the diagram commute
    	    \begin{equation}\label{fibcom}
    	    	\begin{split}
    	    		\xymatrix{
    	    			Z \ar@/_1pc/[ddr]_{\phi_1} \ar@/^1pc/[drr]^{\phi_2} \ar@{-->}[dr]^{u}\\
    	    			&X\cap\mathds{F}_{\mathbf{M}}\ar[d]_{\widetilde{\pi}}\,\ar@{^(->}[r]^{\quad i_2} & \mathds{F}_{\mathbf{M}}\ar[d]^{\pi}\\
    	    			&\mathcal{Y}\,\ar@{^(->}[r]^{i_1} & \EuScript{F}_{\mathbf{M}}
    	    		}
    	    	\end{split}
    	    \end{equation}
    	
    	Let $z\in Z$, then
    	    \begin{equation*}
    	    	(\pi\circ\phi_2)(z)=(i_1\circ\phi_1)(z)\ \Longrightarrow\ \pi\big(\phi_2(z)\big)=i_1\big(\phi_1(z)\big)\ \Longrightarrow\ \phi_2(z)\sbullet T_{\mathbf{M}}=\phi_1(z)\in\mathcal{Y},
    	    \end{equation*}
    	so there exists $x\in X\cap\mathds{F}_{\mathbf{M}}$ such that $\phi_1(z)=x\sbullet T_{\mathbf{M}}$ because it is $T$-invariant. Thus, we can define $u:=\phi_2$, which by definition satisfies $i_2\circ u=\phi_2$.
    	
    	Now, observe that
    	    \begin{eqnarray*}
    	    	(\widetilde{\pi}\circ u)(z)&=&\widetilde{\pi}\big(u(z)\big)\\
    	    	&=&\widetilde{\pi}\big(\phi_2(z)\big)\\
    	    	&=&\pi\big(\phi_2(z)\big)\\
    	    	&=&(\pi\circ\phi_2)(z)\\
    	    	&=&(i_1\circ\phi_1)(z)\\
    	    	&=&i_1\big(\phi_1(z)\big)\\
    	    	&=&\phi_1(z),
    	    \end{eqnarray*}
    	i.e., $\widetilde{\pi}\circ u=\phi_1$ proving that the diagram \eqref{fibcom} commutes. Finally, to verify the uniqueness, suppose there exists $u':Z\longrightarrow X\cap\mathds{F}_{\mathbf{M}}$ such that $i_2\circ u'=\phi_2$ and let $z\in Z$ be arbitrary. Then
    	    \begin{equation*}
    	    	(i_2\circ u')(z)=\phi_2(z)=(i_2\circ u)(z)\quad \Longrightarrow\quad i_2\big(u'(z)\big)=i_2\big(u(z)\big)\quad \Longrightarrow\quad u'(z)= u(z),
    	    \end{equation*}
    	from which $u'=u$.
    	
    	Thus, $X\cap\mathds{F}_{\mathbf{M}}\cong\mathcal{Y} \times_{\EuScript{F}_{\mathbf{M}}}\mathds{F}_{\mathbf{M}}$, and hence $X=\overline{X\cap\mathds{F}_{\mathbf{M}}}\cong\overline{\mathcal{Y} \times_{\EuScript{F}_{\mathbf{M}}}\mathds{F}_{\mathbf{M}}}$.
    \end{proof}
    
    If instead of considering a $T$-invariant subvariety, we consider a scheme-theoretic point, we obtain the next result.
    
    \begin{proposition}\label{schemepoint}
    	Let $\eta\in\EuScript{F}_{\mathbf{M}}$ be a scheme-theoretic point, then $\overline{\overline{\{\eta\}}\times_{\EuScript{F}_{\mathbf{M}}}\mathds{F}_{\mathbf{M}}}$ is $T$-invariant and irreducible.
    \end{proposition}     
        
    \begin{proof}
    	Consider the commutative diagram
    	    \begin{equation*}
    	    	\begin{tikzcd}
    	    		T_{\mathbf{M}}\ar[r] & \overline{\{\eta\}}\times_{\EuScript{F}_{\mathbf{M}}}\mathds{F}_{\mathbf{M}}\ar{d}{\pi_2}\ar{r}{\pi_1} & \overline{\{\eta\}}\ar[hook]{d}{i}\\
    	    		T_{\mathbf{M}}\ar[r] & \mathds{F}_{\mathbf{M}}\ar{r}{\pi} & \EuScript{F}_{\mathbf{M}} 
    	    	\end{tikzcd}
    	    \end{equation*}
    	
    	Note that
    	    \begin{eqnarray*}
    	    	\overline{\{\eta\}}\times_{\EuScript{F}_{\mathbf{M}}}\mathds{F}_{\mathbf{M}}&=&\left\{(y,\Lambda)\in\overline{\{\eta\}}\times\mathds{F}_{\mathbf{M}}\mid\pi(\Lambda)=i(y)\right\}\\
    	    	&=&\left\{(y,\Lambda)\in\overline{\{\eta\}}\times\mathds{F}_{\mathbf{M}}\mid\pi(\Lambda)=y\right\}\\
    	    	&=&\left\{(y,\Lambda)\in\overline{\{\eta\}}\times\mathds{F}_{\mathbf{M}}\mid\Lambda\in\pi^{-1}(y)\right\}\\
    	    	&=&\bigcup_{y\in\overline{\{\eta\}}}\left(\{x\}\times\pi^{-1}(y)\right)\\
    	    	&\cong&\bigcup_{y\in\overline{\{\eta\}}}\pi^{-1}(y),
    	    \end{eqnarray*}
    	from which it follows that $\overline{\{\eta\}}\times_{\EuScript{F}_{\mathbf{M}}}\mathds{F}_{\mathbf{M}}$ is a set of orbits parametrized by $\overline{\{\eta\}}$, implying that $\overline{\{\eta\}}\times_{\EuScript{F}_{\mathbf{M}}}\mathds{F}_{\mathbf{M}}$ is $T$-invariant. Consequently, $\overline{\overline{\{\eta\}}\times_{\EuScript{F}_{\mathbf{M}}}\mathds{F}_{\mathbf{M}}}$ is also $T$-invariant. Since $\mathds{F}_{\mathbf{M}}\cong\EuScript{F}_{\mathbf{M}}\times T_{\mathbf{M}}$, the morphism $\pi$ is flat. As flatness is stable under base change, it follows that $\pi_1$ is flat as well. Consequently, all the fibers of $\pi_1$ are equidimensional and irreducible, since they are isomorphic to $T_{\mathbf{M}}$. By \cite[Chap. III, Cor. 9.6, p. 257]{Har77}, it follows that $\overline{\{\eta\}}\times_{\EuScript{F}_{\mathbf{M}}}\mathds{F}_{\mathbf{M}}$ is pure-dimensional, and hence, by part a) of Musta\c t\u{a}'s Proposition in \cite[p. 1]{Mus09}, it is irreducible. Therefore, $\overline{\overline{\{\eta\}}\times_{\EuScript{F}_{\mathbf{M}}}\mathds{F}_{\mathbf{M}}}$ is also irreducible.
    \end{proof}
    
\subsection{Homology of $T$-equivariant subvarieties}\label{subsec4.1}
    Let $\mathds{F}_{d_1<\cdots<d_k}(\mathds{C}^n)^T$ be the set consisting of scheme-theoretic points of $\mathds{F}_{d_1<\cdots<d_k}(\mathds{C}^n)$ which are $T$-invariant. Let $\mathbf{M}\in\EuScript{M}^{d_1,\ldots,d_k}_n$ be a plurimatroid and consider the map
        \begin{equation*}
        	\begin{array}{cccc}
        		\mathds{V}_{\mathbf{M}}:&\EuScript{F}_{\mathbf{M}}&\longrightarrow&\mathds{F}_{d_1<\cdots<d_k}(\mathds{C}^n)^T\\
        		&y&\longmapsto&\eta
        	\end{array},
        \end{equation*}
    where $\overline{\{\eta\}}=\overline{\pi^{-1}(Y)}$ with $Y=\overline{\{y\}}$. The domain of this map can be extended to the disjoint union of the geometric quotients $\EuScript{F}_{\mathbf{M}}$'s, as we formally state below.
    
    \begin{theorem}\label{mainthm}
    	The map
    	    \begin{equation}\label{surjmap}
    	    	\begin{array}{ccc}
    	    		\displaystyle\bigsqcup_{\mathbf{M}\in\EuScript{M}^{d_1,\ldots,d_k}_n}\EuScript{F}_{\mathbf{M}}&\stackrel{\bigsqcup\mathds{V}_{\mathbf{M}}}{\longrightarrow}&\mathds{F}_{d_1<\cdots<d_k}(\mathds{C}^n)^T
    	    	\end{array}
    	    \end{equation}
    	is surjective.
    \end{theorem}
    
    \begin{proof}
    	Let $\eta\in\mathds{F}_{d_1<\cdots<d_k}(\mathds{C}^n)^T$. Then, $X=\overline{\{\eta\}}$ is $T$-invariant and irreducible, so by Lemma \ref{lemclos}, there exists a unique plurimatroid $\mathbf{M}\in\EuScript{M}^{d_1,\ldots,d_k}_n$ such that $X=\overline{X\cap\mathds{F}_{\mathbf{M}}}$. Let $Y=\pi(X\cap\mathds{F}_{\mathbf{M}})\in\EuScript{F}_{\mathbf{M}}$ and $y$ its generic point; we have
    	    \begin{equation*}
    	    	\overline{\pi^{-1}\left(\overline{\{y\}}\right)}=\overline{\pi^{-1}(Y)}=\overline{\pi^{-1}\big(\pi(X\cap\mathds{F}_{\mathbf{M}})\big)}=\overline{X\cap\mathds{F}_{\mathbf{M}}}=X=\overline{\{\eta\}},
    	    \end{equation*}
    	where the third equality holds because $X\cap\mathds{F}_{\mathbf{M}}$ is a relative $T$-invariant closed set. Hence, $\mathds{V}_{\mathbf{M}}(y)=\eta$ and the map \eqref{surjmap} is surjective, as required.
    \end{proof}
    
    \begin{corollary}\label{vmdecomp}
    	The following decomposition holds:
    	    \begin{equation}
    	    	\mathds{F}_{d_1<\cdots<d_k}(\mathds{C}^n)^T=\bigsqcup_{\mathbf{M}\in\EuScript{M}^{d_1,\ldots,d_k}_n}\mathds{V}_{\mathbf{M}}(\EuScript{F}_{\mathbf{M}}).
    	    \end{equation}
    \end{corollary}
    
    \begin{corollary}\label{hom-map}
    	We have a map of the disjoint union of $T$-orbit spaces to homology, given by
    	    \begin{equation}
    	    	\begin{array}{ccc}
    	    		\displaystyle\bigsqcup_{\mathbf{M}\in\EuScript{M}^{d_1,\ldots,d_k}_n}\EuScript{F}_{\mathbf{M}}&\longrightarrow&H_*(\mathds{F}_{d_1<\cdots<d_k}(\mathds{C}^n),\mathds{Z})\\
    	    		y&\longmapsto&\left[\overline{\mathds{V}_{\mathbf{M}}(y)}\right]
    	    	\end{array}.
    	    \end{equation}
    \end{corollary}
    
    Note that if $y\in\EuScript{F}_{\mathbf{M}}$ is a closed point, then $\mathds{V}_{\mathbf{M}}(y)=\Big[\overline{\Lambda\sbullet T_{\mathbf{M}}}\Big]$; furthermore, if $y$ is the generic point of $\EuScript{F}_{\mathbf{M}}$, then $\mathds{V}_{\mathbf{M}}(y)=\Big[\overline{\mathds{F}_{\mathbf{M}}}\Big]$.
    
    Given a subvariety $Y\subset\mathds{F}_{d_1<\cdots<d_k}(\mathds{C}^n)$, let us denote by $h(Y)\in H_*(\mathds{F}_{d_1<\cdots<d_k}(\mathds{C}^n),\mathds{Z})$ its homology class and let us consider the map
        \begin{equation*}
        	h:\mathds{F}_{d_1<\cdots<d_k}(\mathds{C}^n)^T\longrightarrow H_*(\mathds{F}_{d_1<\cdots<d_k}(\mathds{C}^n),\mathds{Z}).
        \end{equation*}  
    
    \begin{definition}
    	We will call a class $\lambda\in H_*(\mathds{F}_{d_1<\cdots<d_k}(\mathds{C}^n),\mathds{Z})$ a \textbf{prime} $T$-\textbf{class} if there exist $Y\subset\mathds{F}_{d_1<\cdots<d_k}(\mathds{C}^n)^T$ such that $\lambda=h(Y)$.
    \end{definition}
    
    We want to describe the preimage of $\lambda$ in $\mathds{F}_{d_1<\cdots<d_k}(\mathds{C}^n)^T$, i.e., we want to describe $h^{-1}(\lambda)\subset\mathds{F}_{d_1<\cdots<d_k}(\mathds{C}^n)^T$. Thus, let $M^{d_1,\ldots,d_k}_n$ be a complete set of representatives of $\nicefrac{\EuScript{M}^{d_1,\ldots,d_k}_n}{\mathfrak{S}_n}$ and let
        \begin{equation*}
        	\begin{array}{cccc}
        		\theta:&\displaystyle\bigsqcup_{\mathbf{M}\in M^{d_1,\ldots,d_k}_n}\EuScript{F}_{\mathbf{M}}&\longrightarrow&H_*(\mathds{F}_{d_1<\cdots<d_k}(\mathds{C}^n),\mathds{Z})\\
        		&y&\longmapsto&\theta(y):=h\left(\overline{\mathds{V}_{\mathbf{M}}(y)}\right)
        	\end{array}.
        \end{equation*}
    
    For every prime $T$-class $\lambda\in H_*(\mathds{F}_{d_1<\cdots<d_k}(\mathds{C}^n),\mathds{Z})$ and $\mathbf{M}\in M^{d_1,\ldots,d_k}_n$, we define
        \begin{equation*}
        	\EuScript{F}_{\mathbf{M}}(\lambda):=\{y\in\EuScript{F}_{\mathbf{M}}\mid\theta(y)=\lambda\}
        \end{equation*}
    and
        \begin{equation*}
        	M^{d_1,\ldots,d_k}_n(\lambda):=\{\mathbf{M}\in \EuScript{M}^{d_1,\ldots,d_k}_n\mid\EuScript{F}_{\mathbf{M}}(\lambda)\neq\emptyset\}.
        \end{equation*}
    
    \begin{corollary}\label{invtclass}
    	Let $\lambda\in H_*(\mathds{F}_{d_1<\cdots<d_k}(\mathds{C}^n),\mathds{Z})$ a prime $T$-class. Then the set
    	    \begin{equation*}
    	    	\EuScript{F}(d_1<\cdots<d_k,n,\lambda):=\displaystyle\bigsqcup_{\mathbf{M}\in M^{d_1,\ldots,d_k}_n(\lambda)}\mathds{V}_{\mathbf{M}}\big(\EuScript{F}_{\mathbf{M}}(\lambda)\big)
    	    \end{equation*}
    	is naturally a parameter space for $h^{-1}(\lambda)$.
    \end{corollary}    

\section{Thin Schubert cells in the flag variety $\mathds{F}_{1<n-1}(\mathds{C}^n)$}\label{sec5}	
    In the special case of the flag variety $\mathds{F}_{1<n-1}(\mathds{C}^n)$ for $n\geq 3$, it is possible to provide a classification of the thin Schubert cells and to determine their number. To that end, we remember some necessary definitions.
    
     Let $M=([n],\mathcal{B})$ be a matroid of rank $d$ on $n$ elements, i.e., each basis $B\in\mathcal{B}$ has size $d$. Then:
        \begin{itemize}
        	\item (cf. \cite[\S1.1.3, p. 4]{BGW03}) the \textit{rank} of $S\subseteq[n]$, denoted as $r(S)$, is given by
        	\begin{equation}\label{def:ranksubset}
        		r(S)=\max_{B\in\mathcal{B}}{|S\cap B|};
        	\end{equation}
        	\item (cf. \cite[\S4, eq. (1.4.1), p. 25]{Oxl11}) the \textit{closure} of $S\subseteq[n]$, denoted as $\mathrm{cl}(S)$, is the set
        	\begin{equation}\label{def:closuresubset}
        		\mathrm{cl}(S)=\{i\in[n]:r(S\cup\{i\})=r(S)\};
        	\end{equation}
        	\item (cf. \cite[\S1.4, p. 28]{Oxl11}) the set of \textit{flats of} $M$, we denote as $\operatorname{Flats}(M)$, is given by
        	\begin{equation}\label{def:flats}
        		\operatorname{Flats}(M)=\{S\subseteq[n]:\mathrm{cl}(S)=S\}.
        	\end{equation} 
        \end{itemize}
        
    The next definition has been adapted for our purposes of Jarra and Lorscheid \cite[\S0, p. 2]{JL24}.
    
    \begin{definition}\label{def:flagmat}
    	Let $E=[n]$ and $\mathbf{r}=(1,n-1)$. A plurimatroid $\mathbf{M}=(M_1,M_2)\in\EuScript{M}^{1,n-1}_n$ is a \emph{flag matroid of rank $\mathbf{r}$ on $E$} if $\operatorname{Flats}(M_1)\subset\operatorname{Flats}(M_2)$. We also say that $M_1$ \emph{is a quotient of} $M_2$, and write $M_2\twoheadrightarrow M_1$ in this case.
    \end{definition}
    
    \begin{definition}\label{cellcomprop}
    	Let $\mathbf{M}=(M_1,M_2)\in\EuScript{M}^{1,n-1}_n$ be a plurimatroid and
    	    \begin{equation*}
    	    	\mathds{F}_{\mathbf{M}}=(G_{M_1}\times G_{M_2})\cap\mathds{F}_{1<n-1}(\mathds{C}^n)
    	    \end{equation*}
    	its associated thin Schubert cell. We say that $\mathds{F}_{\mathbf{M}}$ is \textbf{complete} if $\mathds{F}_{\mathbf{M}}=G_{M_1}\times G_{M_2}$, and \textbf{proper} if $\mathds{F}_{\mathbf{M}}\subsetneq G_{M_1}\times G_{M_2}$.  
    \end{definition}
    
    \begin{remark}
    	The completeness is equivalent to $G_{M_1}\times G_{M_2}\subseteq\mathds{F}_{1<n-1}(\mathds{C}^n)$, whereas the properness is equivalent to $G_{M_1}\times G_{M_2}\nsubseteq\mathds{F}_{1<n-1}(\mathds{C}^n)$.
    \end{remark}
    
    \begin{lemma}\label{lem:flats}
    	Let $\mathbf{M}=(M_1,M_2)\in\EuScript{M}^{1,n-1}_n$ be a plurimatroid, and denote by $\mathcal{B}_1$ and $\mathcal{B}_2$ the sets of bases of $M_1$ and $M_2$, respectively. Then
    	\begin{equation*}
    		\operatorname{Flats}(M_1)=\{[n]\setminus K_1,[n]\}\quad\text{and}\quad \operatorname{Flats}(M_2)=\{S\subseteq[n]:|K_2\setminus S|\neq1\},
    	\end{equation*}
    	where $K_1:=\{k\in[n]:\{k\}\in\mathcal{B}_1\}$ and $K_2:=\{k\in[n]:[n]\setminus\{k\}\in\mathcal{B}_2\}$. 
    \end{lemma}
    
    \begin{remark}\label{obskis}
    	Note that by definition of sets $K_1$ and $K_2$ in Lemma \ref{lem:flats}, we have $|K_i|=|\mathcal{B}_i|$ for $i=1,2$.
    \end{remark}
    
    \begin{proof}
    	\textsc{Flats of $M_1$}. Let $S\subseteq[n]$ and $k\in K_1$, then
    	    \begin{equation*}
    	    	|S\cap\{k\}|=\begin{cases*}
    	    		1, &if\, $k\in S$\\
    	    		0, &if\, $k\notin S$
    	    	\end{cases*}.
    	    \end{equation*}
    	
    	It follows that
    	    \begin{equation*}
    	    	r(S)=\max_{k\in K_1}|S\cap\{k\}|=\begin{cases*}
    	    		1, &if\, $S\cap K_1\neq\emptyset$\\
    	    		0, &if\, $S\cap K_1=\emptyset$
    	    	\end{cases*}.
    	    \end{equation*}
    	
    	Consequently, if $S\cap K_1=\emptyset$, then $r(S)=0$ and hence $\mathrm{cl}(S)=[n]\setminus K_1$. On the other hand, if $S\cap K_1\neq\emptyset$, then $r(S)=1$ which implies that $r(S\cup\{i\})=1=r(S)$ for all $i\in[n]$; therefore, $\mathrm{cl}(S)=[n]$. This proves that $\operatorname{Flats}(M_1)=\{[n]\setminus K_1,[n]\}$.
    	
    	\textsc{Flats of $M_2$}. Let $S\subseteq[n]$ and $k\in K_2$, then $|S\cap([n]\setminus\{k\})|=|S\setminus\{k\}|$ and hence
    	    \begin{equation*}
    	    	|S\setminus\{k\}|=\begin{cases*}
    	    		|S|, &if\, $k\notin S$\\
    	    		|S|-1, &if\, $k\in S$
    	    	\end{cases*}.
    	    \end{equation*}
    	
    	Thus we have
    	    \begin{equation*}
    	    	r(S)=\max_{k\in K_2}|S\setminus\{k\}|=\begin{cases*}
    	    		|S|, &if\, $K_2\nsubseteq S$\\
    	    		|S|-1, &if\, $K_2\subseteq S$
    	    	\end{cases*}.
    	    \end{equation*}
    	
    	Note that:
    	    \begin{itemize}
    	    	\item If $K_2\subseteq S$, then $r(S\cup\{i\})=|S|=r(S)+1$ for all $i\notin S$; hence, $\mathrm{cl}(S)=S$.
    	    	\item If $K_2\nsubseteq S$, then $r(S)=|S|$, and we have two possibilities:
    	    	    \begin{enumerate}
    	    	    	\item $|K_2\setminus S|=1$. In this case, $K_2\setminus S=\{k_0\}$, where $k_0$ is the unique element of $K_2$ that does not belong to $S$; that is, $K_2\setminus\{k_0\}\subseteq S$, which implies $K_2\subseteq S\cup\{k_0\}$. Hence,
    	    	    	\begin{equation*}
    	    	    		r(S\cup\{k_0\})=|S\cup\{k_0\}|-1=|S|=r(S).
    	    	    	\end{equation*}
    	    	    	Therefore, $\mathrm{cl}(S)=S\cup\{k_0\}\neq S$.
    	    	    	\item $|K_2\setminus S|\geq 2$. In this setting, $K_2\not\subseteq S\cup\{i\}$ for all $i\notin S$. Hence,
    	    	    	\begin{equation*}
    	    	    		r(S\cup\{i\})=|S\cup\{i\}|=|S|+1=r(S)+1,
    	    	    	\end{equation*}
    	    	    	which implies that $i\notin\operatorname{cl}(S)$. Therefore, $\mathrm{cl}(S)=S$.
    	    	    \end{enumerate}
    	    \end{itemize}
    	
    	Consequently, $\operatorname{Flats}(M_2)=\{S\subseteq[n]:|K_2\setminus S|\neq1\}$ as required.
    \end{proof}
    
    \begin{proposition}[Characterization of Thin Schubert Cells]\label{charthincells}
    	With the same notation, hypotheses and assumptions as in Lemma \ref{lem:flats}, consider the flag variety $\mathds{F}_{1<n-1}(\mathds{C}^n)$ and let $\mathds{F}_{\mathbf{M}}$ denote its associated thin Schubert cell. The following statements are equivalent:
    	    \begin{enumerate}
    	    	\item $\mathds{F}_{\mathbf{M}}$ is nonempty.
    	    	\item $\mathbf{M}$ is a flag matroid of rank $\mathbf{r}=(1,n-1)$.
    	    	\item $|K_1\cap K_2|\neq 1$.
    	    \end{enumerate}
    	Moreover:
    	    \begin{itemize}
    	    	\item $\mathds{F}_{\mathbf{M}}$ is complete if and only if $|K_1\cap K_2|=0$.
    	    	\item $\mathds{F}_{\mathbf{M}}$ is proper if and only if $|K_1\cap K_2|\geq 2$.
    	    \end{itemize}   
    \end{proposition}
    
    \begin{proof}
    	Let $(W_1,W_2)\in G(1,n)\times G(n-1,n)$ and
    	    \begin{equation*}
    	    	v_k:=p_{\{k\}}(W_1)\quad\text{and}\quad\alpha_k:=p_{[n]\setminus\{k\}}(W_2)\quad\forall k\in[n].
    	    \end{equation*}
    	Note that
    	    \begin{align*}
    	    	(W_1,W_2)\in\mathds{F}_{\mathbf{M}}&\quad\Longleftrightarrow\quad\sum_{k=1}^{n}(-1)^{k-1}\alpha_kv_k=0 \tag*{(by the incidence equation)}\\
    	    	&\quad\Longleftrightarrow\quad\sum_{k\in K_1\cap K_2}^{n}(-1)^{k-1}\alpha_kv_k=0 \tag{since $(W_1,W_2)\in G_{M_1}\times G_{M_2}$}
    	    \end{align*}
    	    
    	    \begin{itemize}
    	    	\item If $|K_1\cap K_2|=1$, then $\alpha_{k'}v_{k'}=0$ for the unique $k'\in K_1\cap K_2$, which is a contradiction since $\alpha_{k'}$ and $v_{k'}$ are nonzero; hence, $\mathds{F}_{\mathbf{M}}=\emptyset$.
    	    	\item If $|K_1\cap K_2|=0$, then the last equality is satisfied trivially and imposes no constraint on the summands; consequently, $G_{M_1}\times G_{M_2}\subset\mathds{F}_{1<n-1}(\mathds{C}^n)$.
    	    	\item If $|K_1\cap K_2|\geq 2$, then the reduced incidence equation imposes exactly one independent condition for $(W_1,W_2)\in G_{M_1}\times G_{M_2}$ to lie in $\mathds{F}_{1<n-1}(\mathds{C}^n)$; thus $\mathds{F}_{\mathbf{M}}\subsetneq G_{M_1}\times G_{M_2}$.
    	    \end{itemize}
    	
    	To conclude the proof, we need to show that $\mathbf{M}$ is a flag matroid if and only if $|K_1\cap K_2|\neq1$. By Lemma \ref{lem:flats}, we have $\operatorname{Flats}(M_1)=\{[n]\setminus K_1,[n]\}$ and $\operatorname{Flats}(M_2)=\{S\subseteq[n]:|K_2\setminus S|\neq1\}$. Now, observe that $[n]\in\operatorname{Flats}(M_2)$ trivially, so it suffices to determine when the remaining set $S=[n]\setminus K_1\in\operatorname{Flats}(M_1)$ also belongs to $\operatorname{Flats}(M_2)$. Then
    	    \begin{equation*}
    	    	[n]\setminus K_1\in\operatorname{Flats}(M_2)\quad \Longleftrightarrow\quad |K_2\setminus([n]\setminus K_1)|\neq1\quad \Longleftrightarrow\quad |K_2\cap K_1|\neq1,
    	    \end{equation*}
    	as required.
    \end{proof}
    
    \begin{corollary}[Dimension of Thin Schubert Cells]\label{dimfm}
    	With the same notation as in Lemma \ref{lem:flats} and hypotheses as in Proposition \ref{charthincells}, the following statements hold:
    	    \begin{enumerate}
    	    	\item If $|K_1\cap K_2|=0$, then $\dim\mathds{F}_{\mathbf{M}}=\dim\,(G_{M_1}\times G_{M_2})=|K_1|+|K_2|-2$.
    	    	\item If $|K_1\cap K_2|\geq 2$, then $\dim\mathds{F}_{\mathbf{M}}=\dim\,(G_{M_1}\times G_{M_2})-1=|K_1|+|K_2|-3$.
    	    \end{enumerate}
    \end{corollary}
    
    \begin{corollary}[Global Counting Formulas for Thin Schubert Cells]\label{numcpe}
    	In the setting of Proposition \ref{charthincells}, let
    	    \begin{equation*}
    	    	C:=\{\mathds{F}_{\mathbf{M}}:\mathds{F}_{\mathbf{M}}\ \text{is complete}\},\quad P:=\{\mathds{F}_{\mathbf{M}}:\mathds{F}_{\mathbf{M}}\ \text{is proper}\}\quad\text{and}\quad E:=\{\mathds{F}_{\mathbf{M}}:\mathds{F}_{\mathbf{M}}\ \text{is empty}\}.
    	    \end{equation*}
    	Then:
    	\begin{enumerate}
    		\item $|C|=3^n-2^{n+1}+1$.
    		\item $|P|=4^n-3^n-n\cdot 3^{n-1}$.
    		\item $|E|=n\cdot 3^{n-1}$.
    	\end{enumerate}
    \end{corollary}
    
    \begin{proof}
    	\begin{enumerate}
    		\item For each $i\in[n]$ there are only three cases: (\textit{a}) $i\in K_1$, (\textit{b}) $i\in K_2$ or (\textit{c}) $i\notin K_1\cup K_2$, but never in both, since $K_1\cap K_2=\emptyset$. Hence there are $3^n$ options, including the cases $K_1=\emptyset$ or $K_2=\emptyset$ which we must exclude. Observe that if $K_1=\emptyset$, then $i\in K_2$ or $i\in[n]\setminus K_2$, yielding $2^n$ possibilities and similarly if $K_2=\emptyset$. Since the case $K_1=K_2=\emptyset$ has been counted twice, we add it back once. Consequently, $|C|=3^n-2^{n+1}+1$.
    		\item Here, for each $i\in[n]$ we have exactly four options, namely: (\textit{a}) $i\in K_1\cap K_2$, (\textit{b}) $i\in K_1\setminus K_2$, (\textit{c}) $i\in K_2\setminus K_1$ or (\textit{d}) $i\notin K_1\cup K_2$. Hence, there are $4^n$ possibilities. We exclude now the case $K_1\cap K_2=\emptyset$, leaving $3^n$ options ---(\textit{b}) - (\textit{d})---. To conclude, we must also exclude the case $|K_1\cap K_2|=1$. The unique element of this intersection can be choosen in $n$ ways, and each of the remaining $n-1$ elements has three possibilities ---again, (\textit{b}) - (\textit{d})---, so $3^{n-1}$ choices. Therefore, $|P|=4^n-3^n-n\cdot 3^{n-1}$.
    		\item This case has already been done at the end of the preceding item.
    	\end{enumerate}
    \end{proof}
    
    \begin{corollary}[Restricted Counting Formulas for Thin Schubert Cells]\label{numcpe-res}
    	Under the hypotheses of Proposition \ref{charthincells} and notation of Corollary \ref{numcpe}, let $i,j\in[n]$ and define
    	\begin{align*}
    		C(i,j)&:=\{\mathds{F}_{\mathbf{M}}\in C:|\mathcal{B}_1|=i\ \ \text{and}\ \  |\mathcal{B}_2|=j\},\\
    		P(i,j)&:=\{\mathds{F}_{\mathbf{M}}\in P:|\mathcal{B}_1|=i\ \ \text{and}\ \  |\mathcal{B}_2|=j\},\\
    		E(i,j)&:=\{\mathds{    F}_{\mathbf{M}}\in E:|\mathcal{B}_1|=i\ \ \text{and}\ \  |\mathcal{B}_2|=j\}.
    	\end{align*}   
    	Then:
    	\begin{enumerate}
    		\item $|C(i,j)|=\displaystyle{n\choose i}{n-i\choose j}$.
    		\item $|P(i,j)|=\displaystyle{n\choose i}\left[{n\choose j}-{n-i\choose j}-i{n-i\choose j-1}\right]$.
    		\item $|E(i,j)|=\displaystyle i{n\choose i}{n-i\choose j-1}$.
    	\end{enumerate}
    \end{corollary}
    
    \begin{proof}
    	For each $i\in\{1,2\}$, fix a set $\mathcal{B}_i$ of bases of $M_i$ with $|\mathcal{B}_i|=i$. By Remark \ref{obskis}, we have $|K_i|=|\mathcal{B}_i|$ for $i=1,2$.
    	    \begin{enumerate}
    	    	\item The subsets $K_2$ of $[n]$ disjoint of $K_1$ must be chosen from $n-i$ elements of $[n]\setminus K_1$, in $\displaystyle{n-i\choose j}$ ways. Since there are $\displaystyle{n\choose i}$ choices for $K_1$, the rule product gives $|C(i,j)|=\displaystyle{n\choose i}\displaystyle{n-i\choose j}$.
    	    	\item The total number of subsets $K_2$ of $[n]$ is $\displaystyle{n\choose j}$; among them, those disjoint from $K_1$ are $\displaystyle{n-i\choose j}$. Those with $|K_1\cap K_2|=1$ are obtained choosing the common element in $i$ ways and completing $K_2$ with $j-1$ elements from $[n]\setminus K_1$, in $\displaystyle{n-i\choose j-1}$ ways; hence, $\displaystyle i{n-i\choose j-1}$ in total. Therefore,the number of subsets $K_2$ with $|K_1\cap K_2|\geq 2$ is $\displaystyle{n\choose j}-{n-i\choose j}-i{n-i\choose j-1}$. Multiplying by the $\displaystyle{n\choose i}$ choices of $K_1$, we conclude that $|P(i,j)|=\displaystyle{n\choose i}\left[{n\choose j}-{n-i\choose j}-i{n-i\choose j-1}\right]$.
    	    	\item This case has already been done at the end of the preceding item.
    	    \end{enumerate}    
    \end{proof}
    
    In Appendix \ref{app:codes}, we provide the codes implementing the formulas for the global and restricted counting of thin Schubert cells, for arbitrary $n\geq 3$, which are available for the software \textit{Maple} (see \ref{subapp:maple}) and \textit{Mathematica} (see \ref{subapp:mathematica}).
    
\section{Conclusions}	
    We introduced in \S\ref{subsec2.1} the \emph{plurimatroids}, conceived as Cartesian products of matroids, and used them to define \emph{thin Schubert cells} on flag varieties. From this viewpoint, we extended to these varieties several results previously established by Elizondo et al. for Grassmannians. In \S\ref{sec3} we proved that such thin cells are affine and locally closed subsets of the flag variety and, therefore, they induce a decomposition of it. Pursuing the same combinatorial approach, in \S\ref{sec4} we established the existence of algebro-geometric quotients ---categorical, geometric and universal--- for certain actions of algebraic tori. Furthermore, in \S\ref{subsec4.1} we obtained, among other consequences, a decomposition of the set of scheme-theoretic points of the flag variety that are invariant under the action of a torus. Finally, in \S\ref{sec5} we characterized all nonempty thin Schubert cells in the case of the flag variety $\mathds{F}_{1<n-1}(\mathds{C}^n)$ for $n\geq 3$, and moreover, we obtained explicit formulas for counting all of them. The techniques developed in this work suggest several avenues for future research, including possible extensions to generalized or symplectic flag varieties, and even to Schubert varieties.
    
\appendix
\section{Programming codes}\label{app:codes} 
    \subsection{\textit{Maple} (\textit{v. 2024.2})}\label{subapp:maple}
    \begin{verbatim}
     with(DocumentTools:-Layout):
        	
     thincells := proc(n::posint) 
       local c, e, p, i, j, Zone, CellText, ZoneColor, header, W, S2, S3, S4, S5:
        	
        c := (i, j) -> binomial(n, i)*binomial(n - i, j): 
        e := (i, j) -> i*binomial(n, i)*binomial(n - i, j - 1): 
        p := (i, j) -> binomial(n, i)*binomial(n, j) - c(i, j) - e(i, j):
        	
     Zone := proc(i, j)::integer; 
       option remember; 
          if j <= n - i and (i = 1 or j = 1) then return 1: 
            elif j <= n - i and 2 <= i and 2 <= j then return 2: 
            elif i = 1 and j = n or i = n and j = 1 then return 3: 
            elif j = n - i + 1 and 2 <= i and i <= n - 1 then return 4: 
            elif n - i + 2 <= j then return 5: 
          end if;   
     end proc: 
        	
     CellText := proc(i, j)::string; 
       local z; 
        z := Zone(i, j): 
         if z = 1 then 
             return cat("C = ", c(i, j), ", E = ", e(i, j)): 
           elif z = 2 then 
             return cat("C = ", c(i, j), ", P = ", p(i, j), ", E = ", e(i, j)): 
           elif z = 3 then 
             return cat("E = ", n): 
           elif z = 4 then 
             return cat("P = ", binomial(n, i)*(binomial(n, i - 1) - i), ", E = ", 
                        i*binomial(n, i)): 
           else 
             return cat("P = ", binomial(n, i)*binomial(n, j)): 
         end if; 
     end proc:
        	
     ZoneColor := proc(i, j) 
       local z; 
        z := Zone(i, j): 
         if z = 1 then return "#90c590": 
           elif z = 2 then return "#ffb6c1": 
           elif z = 3 then return "#ff5a42": 
           elif z = 4 then return "#a680ff": 
           else return "#6ea8fe": 
         end if; 
     end proc:
        	
     header := Row(Cell("", fillcolor = "White"), 
                   seq(Cell(j, fillcolor = "LightGray"), j = 1 .. n)):
        	
     W := Table(header, seq(Row(Cell(i, fillcolor = "LightGray"), 
                seq(Cell(Textfield(CellText(i, j), style = Text), 
                fillcolor = ZoneColor(i, j)), j = 1 .. n)), i = 1 .. n), 
                alignment = center, widthmode = pixels, width = 900): 
        	
     S1 := Textfield(cat("# COMPLETE  CELLS  =  ", 3^n - 2^(n + 1) + 1), 
                     style = Text, alignment = centered);
     S2 := Textfield(cat("# PROPER  CELLS  =  ", 4^n - 3^n - n*3^(n - 1)), 
                     style = Text, alignment = centered);
     S3 := Textfield(cat("# NONEMPTY  CELLS  =  ", (2^n - 1)^2 - n*3^(n - 1)), 
                     style = Text, alignment = centered);
     S4 := Textfield(cat("# EMPTY  CELLS  =  ", n*3^(n - 1)), 
                     style = Text, alignment = centered);
     S5 := Textfield(cat("# POSSIBLE  CELLS  =  ", (2^n - 1)^2), 
                     style = Text, alignment = centered);
        	
     DocumentTools:-InsertContent(Worksheet(
      Group(W, Textfield("", style = Text), S1, S2, S3, S4, S5)));
        	
     NULL: 
     end proc:   
    \end{verbatim}
        
\newpage
    \subsection{\textit{Mathematica} (\textit{v. 14.2})}\label{subapp:mathematica}\hphantom{}
    
    \begin{lstlisting}[language=Mathematica,gobble=8]
    	ClearAll[thincells, hexRGB, c, e, p, zone, cellText, zoneColor]; 
    	
    	hexRGB[s_String] := Module[{t = StringReplace[s, "#" -> ""]}, 
    	With[{r = FromDigits[StringTake[t, 2], 16], 
    	      g = FromDigits[StringTake[t, {3, 4}], 16],  
    	      b = FromDigits[StringTake[t, {5, 6}], 16]}, 
    	RGBColor[r/255., g/255., b/255.]]];
    	
    	c[n_, i_, j_] := Binomial[n, i] Binomial[n - i, j]; 
    	e[n_, i_, j_] := i Binomial[n, i] Binomial[n - i, j - 1]; 
    	p[n_, i_, j_] := Binomial[n, i] Binomial[n, j] - c[n, i, j] 
    	                 - e[n, i, j]; 
    	
    	zone[n_, i_, j_] /; (1 <= i <= n && 1 <= j <= n) := 
    	  Which[j <= n - i && (i == 1 || j == 1), 1, j <= n - i && 
    	        i >= 2 && j >= 2, 2, (i == 1 && j == n) || (i == n 
    	        && j == 1), 3, j == n - i + 1 && 2 <= i <= n - 1, 4, 
    	        j >= n - i + 2, 5];
    	
    	cellText[n_, i_, j_] := Module[{z = zone[n, i, j], parts = {}}, 
    	  If[MemberQ[{1, 2}, z], 
    	     AppendTo[parts, Row[{"C = ", c[n, i, j]}]]]; 
    	  If[MemberQ[{2, 4, 5}, z], 
    	     AppendTo[parts, 
    	  Row[{"P = ", Which[z == 2, p[n, i, j], z == 4, 
    	        Binomial[n, i] Binomial[n, j] - e[n, i, j], True, 
    	        Binomial[n, i] Binomial[n, j]]}]];]; 
    	If[MemberQ[{1, 2, 3, 4}, z], 
    	   AppendTo[parts, Row[{"E = ", e[n, i, j]}]]]; 
    	   Row@Riffle[parts, ", "]];
    	
    	zoneColor[z_] := Switch[z, 1, hexRGB["#90c590"], 2, 
    	        hexRGB["#ffb6c1"], 3, hexRGB["#ff5a42"], 4, 
    	        hexRGB["#a680ff"], 5, hexRGB["#6ea8fe"],_, White];
    	
    	thincells[n_Integer] := Module[{jBand = LightGray, 
    	    hdr = LightBlue, top, second, rows, i, j},
    	top = Join[{Item["", Background -> White]}, {Item["|B_M2|", 
    	    Alignment -> Center, Background -> jBand]}, 
    	    Table[SpanFromLeft, {n - 1}]]; 
    	second = Join[{Item["|B_M1|", Alignment -> Center, 
    	       Background -> jBand]}, Table[Item[j, Background -> hdr], 
    	       {j, 1, n}]]; 
    	rows = Table[Join[{Item[i, Background -> hdr]}, 
    	         Table[Item[cellText[n, i, j], Background -> 
    	         zoneColor[zone[n, i, j]]], {j, 1, n}]], {i, 1, n}];        
    	
    	table = Grid[Join[{top, second}, rows], Alignment -> Center, 
    	     Spacings -> {1.2, 0.9}, ItemSize -> Automatic, Frame -> 
    	     All, BaseStyle -> {FontFamily -> "Helvetica", 12}];
    	
    	stats = Column[{Row[{Style["# Complete cells = ", Bold], 
    			TraditionalForm[3^n - 2^(n + 1) + 1]}], 
    		Row[{Style["# Proper cells = ", Bold], 
    			TraditionalForm[4^n - 3^n - n*3^(n - 1)]}],
    		Row[{Style["# Nonempty cells = ", Bold], 
    			TraditionalForm[(2^n - 1)^2 - n*3^(n - 1)]}], 
    		Row[{Style["# Empty cells = ", Bold], 
    			TraditionalForm[n*3^(n - 1)]}],
    		Row[{Style["# Possible cells = ", Bold], 
    			TraditionalForm[(2^n - 1)^2]}]}, Spacings -> 0.6];
    	
    	Column[{table, Spacer[8], stats}]];    
    \end{lstlisting}
                
    \bibliographystyle{alpha}
    \bibliography{References} 
\end{document}